\documentclass[11pt]{amsart} 
\usepackage{xspace,amssymb,amsmath,amscd,amsthm,epsfig,etoolbox,mathtools,color}
\usepackage{booktabs}
\usepackage{comment}
\usepackage{cite,url}

\newlength{\defbaselineskip}
\theoremstyle{plain}
\newtheorem{theorem}{Theorem}
\newtheorem{lemma}[theorem]{Lemma}
\newtheorem{prop}[theorem]{Proposition}
\newtheorem{corollary}[theorem]{Corollary}

\theoremstyle{definition}

\newtheorem{example}[theorem]{Example}
\newtheorem{remark}[theorem]{Remark}

\newcommand{\rev}[1]{\overleftarrow{#1}}

\newcommand{\Z}{\mathbb{Z}}
\newcommand{\Q}{\mathbb{Q}}
\newcommand{\R}{\mathbb{R}}
\newcommand{\N}{\text{N}}
\newcommand{\E}{\text{E}}
\newcommand{\mcG}{\mathcal{G}}

\newcommand{\mcT}{\mathcal{T}}
\newcommand{\ch}{D}
\newcommand{\chmax}{D_{\text{max}}}

\newcommand{\ee}{\mathbf{e}}
\newcommand{\LL}{\mathbf{L}}
\newcommand{\ww}{\mathbf{w}}
\DeclareMathOperator{\area}{area}
\DeclareMathOperator{\AV}{AV}
\DeclareMathOperator{\Cat}{Cat}
\DeclareMathOperator{\dinv}{dinv}
\DeclareMathOperator{\lvl}{lvl}
\DeclareMathOperator{\Pic}{Pic}
\DeclareMathOperator{\pos}{pos}
\DeclareMathOperator{\run}{run}
\DeclareMathOperator{\skc}{SKC}
\DeclareMathOperator{\skf}{SKF}
\DeclareMathOperator{\skp}{SKP}
\DeclareMathOperator{\skv}{SKV}
\DeclareMathOperator{\stat}{stat}
\DeclareMathOperator{\wt}{wt}

\newcommand{\Dmngen}{\mathcal{D}_{n}^m}
\newcommand{\Dmn}[2]{\mathcal{D}_{#2}^{#1}}
\newcommand{\Dmnkgen}{\mathcal{D}_{n}^{m,k}}
\newcommand{\Dmnk}[3]{\mathcal{D}_{#2}^{#1,#3}}
\newcommand{\Cmnk}[3]{C_{#2}^{#1,#3}}
\newcommand{\Cmnkgen}{C_{n}^{m,k}}
\newcommand{\Cmngen}{C_{n}^m}
\newcommand{\Cmn}[2]{C_{#2}^{#1}}
\newcommand{\phibare}{\varphi}
\newcommand{\phimnk}[3]{\varphi_{#2}^{#1,#3}}
\newcommand{\phimnkgen}{\varphi_{n}^{m,k}}
\newcommand{\psibare}{\psi}
\newcommand{\psimnk}[3]{\psi_{#2}^{#1,#3}}
\newcommand{\psimnkgen}{\psi_{n}^{m,k}}

\newcommand{\zcomm}{\llap{\phantom{,}}}

\begin{document}

\title{Skeletal generalizations of Dyck paths, parking functions, and chip-firing games}

\thanks{This work was supported by grants from the Simons
  Foundation/SFARI (\#854037, S.B.; \#633564, N.A.L.; \#429570,
  G.S.W.); by NSF Grant DMS-2246967 (S.B.); and by internal University
  of Vermont funding from the Office of the Vice President for Research
  (C.C., M.O.) and from the Department of Mathematics \& Statistics
  (P.M.)}

\author[]{Spencer Backman}
\address{Dept. of Mathematics and Statistics, University of Vermont, Burlington, VT 05401}
\email{spencer.backman@uvm.edu}

\author[]{Cole Charbonneau}
\email{colecharbonneau@gmail.com}

\author[]{Nicholas A. Loehr}
\address{Dept. of Mathematics, Virginia Tech, Blacksburg, VA 24061-0123 }
\email{nloehr@vt.edu}

\author[]{Patrick Mullins}
\email{patricksmichaelsmullins@gmail.com}

\author[]{Mazie O'Connor}
\address{Dept. of Mathematics and Statistics, University of Vermont, Burlington, VT 05401}
\email{mazie.oconnor@uvm.edu}

\author[]{Gregory S. Warrington}
\address{Dept. of Mathematics and Statistics, University of Vermont, Burlington, VT 05401}
\email{gregory.warrington@uvm.edu}

\begin{abstract}
  For $0\leq k\leq n-1$, we introduce a family of \emph{$k$-skeletal
  paths} which are counted by the $n$\nobreakdash-th Catalan number for each $k$,
  and specialize to Dyck paths when $k=n-1$.  We similarly introduce
  \emph{$k$-skeletal parking functions} which are equinumerous with
  the spanning trees on $n+1$ vertices for each $k$, and specialize to
  classical parking functions for $k=n-1$. The preceding constructions
  are generalized to paths lying in a trapezoid with base $c > 0$ and
  southeastern diagonal of slope $1/m$; $c$ and $m$ need not be
  integers. We give bijections among these families when $k$ varies
  with $m$ and $c$ fixed.  Our constructions are motivated by chip
  firing and have connections to combinatorial representation theory
  and tropical geometry.
\end{abstract}

\maketitle

 

\section{Introduction}
\label{sec:intro} 

Motivated by developments in chip firing, tropical geometry,
and combinatorial representation theory, we introduce many new families
of combinatorial objects, called \emph{$k$-skeletal paths}
and \emph{$k$-skeletal functions}, that depend on $k$ and certain 
other parameters. These objects generalize Dyck paths, parking functions,
and lattice paths inside trapezoids and thereby provide new combinatorial
interpretations of Catalan numbers, ballot numbers, parking function
counts, and $q$-analogues of these numbers. 
As an initial special case of these ideas, 
we describe our generalizations of Dyck paths of order $n$. Given $k,n$ 
with $0\leq k<n$, define a \emph{$k$-skeletal path} to be a path from $(0,0)$ 
to $(n,n)$ consisting of unit-length east steps and north steps satisfying 
these two conditions:
\begin{enumerate}
  \item[(K1)] The last $k+1$ north steps start weakly to the left of 
  the line $x=y$. 
  \item[(K2)] There do not exist $k+1$ consecutive rows in which
    the north steps all start strictly to the left of the line $x=y$.
\end{enumerate}
Our first main result is Theorem~\ref{thm:main-skp}, which constructs
bijections between the set of $k$-skeletal paths and the set of
$k'$-skeletal paths for all $k,k'$ between $0$ and $n-1$.  Because the
$(n-1)$-skeletal paths are the same as Dyck paths, we see that the
number of $k$-skeletal paths is given by the $n$\nobreakdash-th
Catalan number. These $k$-skeletal paths provide combinatorial
interpretations of the Catalan numbers that we believe to be new (see
Stanley's compilations of interpretations~\cite[Ex. 6.19]{ECII}
and~\cite{addendum}).  In fact, Theorem~\ref{thm:main-skp} and our
subsequent results apply in a much more general setting.  In lattice
path enumeration theory, one may study classical Dyck paths (paths in
a triangle with boundary $x=y$) or rational-slope Dyck paths (paths in
a triangle with boundary $x=(a/b)y$) or trapezoidal paths (paths in
the trapezoid bounded by $y=0$, $y=n$, $x=0$, and $x=my+c$).  Some of
the recent literature in this area includes~\cite{loehr-trap,
  armstrong-ratcat,armstrong2013rational,gorsky-mazin-ii}.
Section~\ref{sec:skel-paths} defines our general notion of
$k$-skeletal paths based on an additive subgroup $\mcG$ of $\R$ and
parameters $c,m\in\mcG$ and $n\in\Z_{>0}$.  Informally, these
$k$-skeletal paths generalize paths in a trapezoid with height $n$,
base $c$, and diagonal of slope $1/m$, where all east steps in the
paths have lengths in $\mcG$.  Classical lattice paths arise as the
special case where $\mcG=\Z$ and $m,c$ are integers.  For that special
case, Corollary~\ref{cor:main-enum} states that the number of
$k$-skeletal paths (for each $k$ between $0$ and $n-1$) is given by an
$m$-ballot number. Section~\ref{sec:first-return} uses the well-known
``first-return recursion'' for $m$-Dyck paths to reprove this
enumerative result in the special case $c=1$.

In Section~\ref{sec:skel-fns}, we extend the framework of
Section~\ref{sec:skel-paths} to consider $k$-skeletal
functions, which can be viewed as lattice paths with north steps
labeled according to certain rules. Our $k$-skeletal labeled
lattice paths generalize various kinds of parking functions 
(determined by the parameters $n,m,c,\mcG$). For fixed choices
of these parameters, the number of $k$-skeletal labeled paths
is independent of $k$ and equals the corresponding parking function count.
See Theorem~\ref{thm:main-skf} for the precise statement.

While Sections~\ref{sec:skel-paths} and~\ref{sec:skel-fns} focus on
the enumerative aspects of $k$-skeletal objects, the next two
sections explore connections in two different directions. 
Section~\ref{sec:skel-chip} views the results of
Sections~\ref{sec:skel-paths} and~\ref{sec:skel-fns} through the lens
of chip firing on graphs. In fact, this subject furnished the initial
motivation for our combinatorial definitions of $k$-skeletal objects.
We are led naturally from classical chip firing to the notion of
\emph{$\mcG$-valued chip firing}, where $\mcG$ is any additive subgroup of
$\R$. Section~\ref{sec:t-analog} considers $t$-analogues of $k$-skeletal
objects using statistics with close connections to various $\dinv$
statistics~\cite{hag-book} appearing in the theory of $q,t$-Catalan polynomials.
This raises the possibility that our $k$-skeletal objects may yield
useful new insights on $q,t$-Catalan polynomials, diagonal harmonics modules,
and related constructions in combinatorial representation theory.

The remainder of this introduction consists of two independent subsections
giving more detailed background from chip firing to motivate our combinatorial
results.  Section~\ref{subsec:kskel} provides an introductory
account of the general role chip firing plays in a story that, in
Sections~\ref{sec:skel-paths} and~\ref{sec:skel-fns}, is told primarily
through the language of lattice paths. Section~\ref{subsec:motivation}
dives deeper into the algebraic and topological developments in chip-firing 
theory that led us to the definitions of $k$-skeletal paths and functions
given here. The rest of this paper is mostly independent
of the next two subsections, which nevertheless provide context for
understanding the connections between our results and other areas.

\subsection{Motivation from Classical Chip Firing}
\label{subsec:kskel}

The chip-firing game \cite{bjorner-lovasz-shor, biggs}
(see~\cite{corry2018divisors,klivans} for a general introduction, and
additional references below) is a dynamical process that can be
played on any loopless graph such as the complete graph $K_{1+n}$ on
the vertex set $\{0,1,2,\ldots,n\}$.  In the game, we start with a
\emph{chip configuration} (or \emph{divisor}) $\ch =
(\ch(1),\ldots,\ch(n))^T$, encoded as a column vector and interpreted
as the placement of $\ch(i)\geq 0$ chips on vertex $i$ for each $i$
between $1$ and $n$. A vertex $v \geq 1$ \emph{fires} by moving one
chip from $v$ to each of its $n$ neighbors.  Vertex $0$ is
distinguished as the \emph{sink} and is not allowed to fire.  We
typically focus on the \emph{legal firings} --- those for which no
non-sink vertex ends up with a negative number of chips after the
firing action. Configurations without any legal firings are
\emph{stable}.  Several legal firings are illustrated in
Figure~\ref{fig:firing}.

Because every firing move sends one chip to the sink, any sequence of
legal firings must eventually terminate. The terminal chip
configurations resulting from initial configurations in which every
vertex can fire are called \emph{critical configurations} (also called
\emph{recurrent configurations}). Critical configurations play an
important role in chip-firing theory, as outlined in
Section~\ref{subsec:motivation}.
\begin{figure}[t]
\begin{center}
  \includegraphics[width=1.0\linewidth]{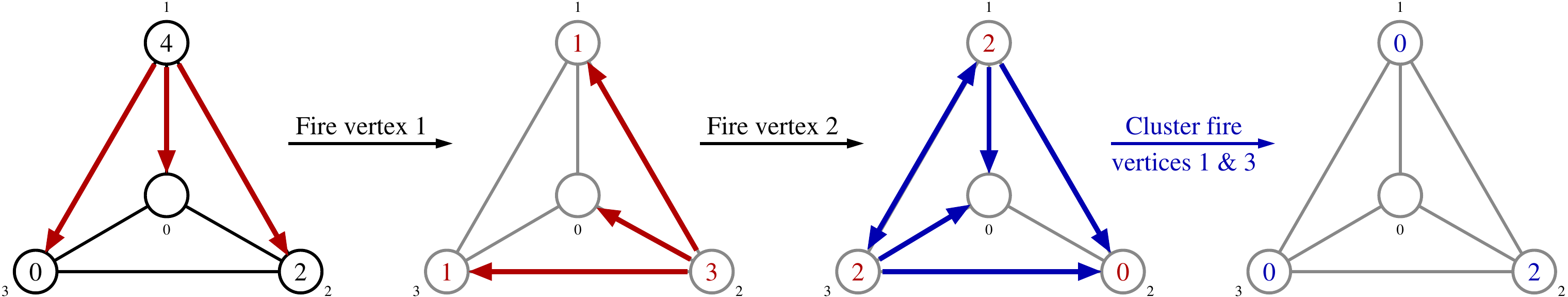}
  \caption{The leftmost graph illustrates $K_{1+3}$ with a chip
    distribution of $\ch(1)=4$, $\ch(2)=2$ and $\ch(3)=0$. Vertices
    $1$ and $2$ are fired in sequence. The resulting chip
    configuration is stable, but vertices $1$ and $3$ can be fired as
    a cluster to obtain the fourth distribution, which is
    superstable.}
\label{fig:firing}
\end{center}
\end{figure}
For $n=3$, there are $(3+1)^{3-1}=16$ different
critical configurations. While we could list all sixteen as
vertex-weighted graphs, it is convenient to introduce a
method of associating an unlabeled lattice path with any chip
configuration. Associate with the configuration $\ch$ the path
$\pi(\ch)$ that has $j$ unit-length north steps in column $i$ when $j$
non-sink vertices have exactly $i$ chips on them. These north steps
are arranged vertically so that we obtain a path from $(0,0)$ to
$(n,n)$ by connecting the runs of north steps in each column by
unit-length east steps as necessary. The sixteen critical
configurations for $K_{1+3}$ give rise to the five unlabeled lattice
paths shown in Figure~\ref{fig:paths}(a). This mapping from
arbitrary chip configurations to unlabeled paths is many-to-one, but
we can modify it by adding labels to get a bijection between chip
configurations and labeled lattice paths. To obtain a labeled lattice
path, we label the north steps along each line $x=i$ with the elements
of $\ch^{-1}(\{i\})$, sorted into increasing order from bottom to
top. Figure~\ref{fig:paths}(b) illustrates a collection of three
configurations sharing the same underlying unlabeled path.

\begin{figure}[h]
\begin{center}
  \includegraphics[width=0.9\linewidth]{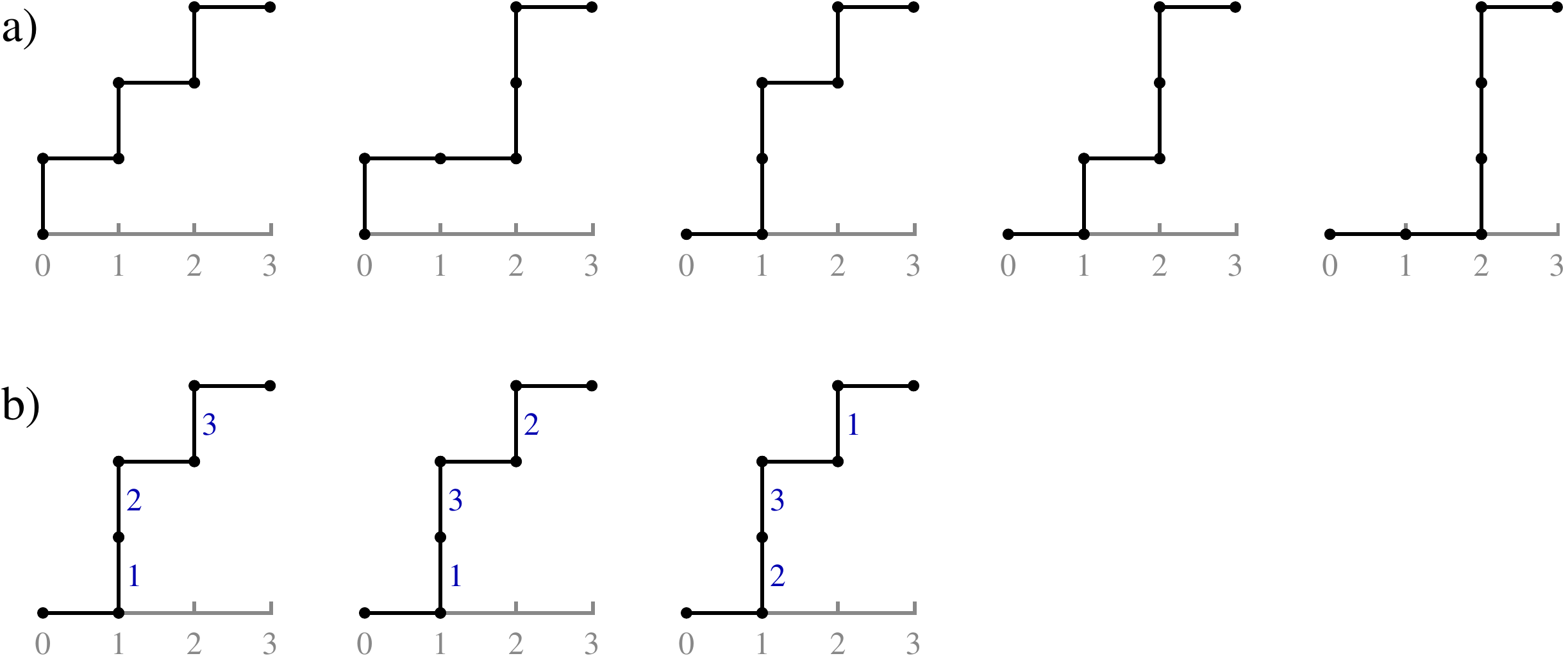}
  \caption{(a) The five unlabeled lattice paths representing the sixteen
    critical configurations for $n=3$. In (b) we show the three labeled
    lattice paths sharing the same unlabeled lattice path
    corresponding to all configurations in which one non-sink vertex
    has two chips and the other two have one chip each.}
\label{fig:paths}
\end{center}
\end{figure}

What we have described so far --- with only one vertex firing at a
time --- is the \emph{abelian sandpile model}~\cite{bak,Dhar}.  More
generally, we can \emph{cluster fire} a subset $S\subseteq
\{1,2,\ldots,n\}$ by firing all of the vertices in $S$ simultaneously.
In the \emph{unconstrained firing model}, any nonempty subset of
vertices is allowed to cluster fire; the cluster firing is
\emph{legal} if no vertex ends up with a negative number of chips. A
\emph{superstable} configuration is one in which no nonempty subset of
the non-sink vertices can legally cluster fire. The third
configuration in Figure~\ref{fig:firing} is stable but not
superstable, as the two vertices $1$ and $3$ can legally be fired
simultaneously. For $n=3$, there are sixteen superstable
configurations, which are represented by the unlabeled paths shown in
Figure~\ref{fig:superstable}.

\begin{figure}[h]
\begin{center}
  \includegraphics[width=0.9\linewidth]{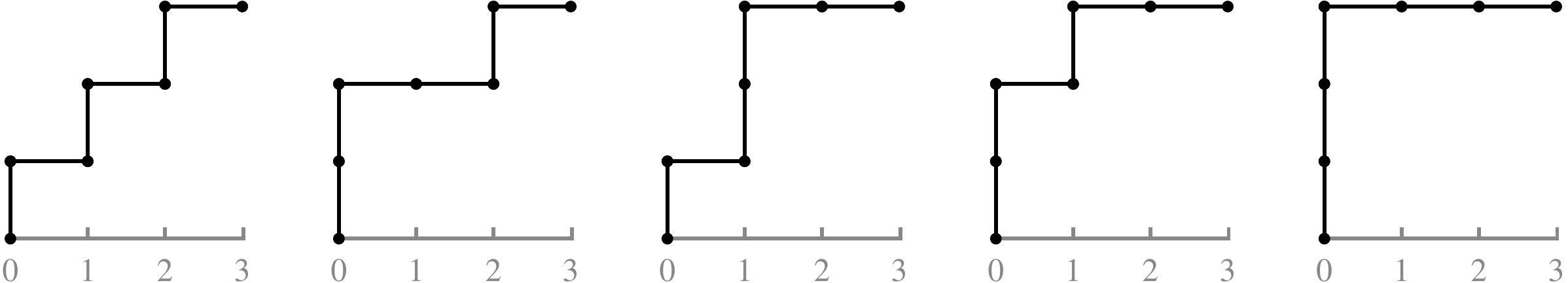}
  \caption{The five unlabeled paths corresponding to the sixteen superstable 
configurations for $n=3$.}
\label{fig:superstable}
\end{center}
\end{figure}

There are two aspects of the superstable configurations we wish to
highlight. First, there is a duality between the critical and
superstable configurations, as suggested by the fact that there are
sixteen configurations of each type for $n=3$. More precisely, let
$\chmax$ denote the configuration in which each non-sink vertex of 
$K_{1+n}$ has $n-1$ chips. This is the maximal chip configuration that is
stable. It turns out that $\ch$ is critical if and only if
$\chmax-\ch$ is superstable. We refer the reader
to~\cite[Thm. 2.6.19]{klivans} for a proof. The result in the particular
case $n=3$ can be seen by examining Figures~\ref{fig:paths}(a)
and~\ref{fig:superstable}. Also see Remark~\ref{rmk:rotate} 
in Section~\ref{sec:skel-paths}.

The second aspect worth highlighting is that superstable
configurations are parking functions. Under our labeling conventions,
a function $f:\{1,2,\ldots,n\}\rightarrow \{0,1,\ldots,n-1\}$ is a 
\emph{parking function} if for all $i$ between $0$ and $n-1$,
the number of $j$ with $f(j)\geq i$ is at most $n-i$
(see~\cite[Chap. 5]{hag-book} and Section~\ref{sec:skel-fns}). 
If we associate a labeled path with each
parking function, as we did above for chip configurations, one can
check that the conditions on the values of a parking function
translate into the underlying unlabeled path being a Dyck path.
In fact, this is just a rephrasing of the conditions for a
configuration to be superstable: no subset $S$ of $i$ non-sink
vertices, where $1\leq i\leq n$, has at least $n-i+1$ chips on each
vertex. Parking functions play an important role in combinatorics and
representation theory. (For instance, the Frobenius series for the
diagonal coinvariants can be expressed as a weighted sum indexed
by parking functions~\cite{carlsson,hhlru}.)
As such, any generalization of superstable
configurations has the potential to illuminate questions in
combinatorial representation theory.

Looking back at Conditions~(K1) and~(K2) defining the $k$-skeletal
paths, we may interpret the two extreme cases using chip firing. On one hand,
the $0$-skeletal paths are the paths corresponding to critical configurations, 
arising naturally from the abelian sandpile model. 
On the other hand, the $(n-1)$-skeletal paths are the Dyck paths
corresponding to the superstable configurations, arising naturally
from the unconstrained firing model. Figure~\ref{fig:ex3} illustrates
all $k$-skeletal paths when $n=3$ and $k$ is $0$, $1$, or $2$.
Figure~\ref{fig:rotation} illustrates the duality between $0$-skeletal
paths and $(n-1)$-skeletal paths for a more general choice of parameters.

Conditions (P0)--(P2) of Section~\ref{subsec:path-def} define
$k$-skeletal paths in our more general setup involving parameters
$m,c$ chosen from an additive subgroup $\mcG$ of $\R$.
These conditions provide a way to define
$k$-skeletal paths associated with general trapezoidal or triangular
regions, as specified by the parameters $n$, $m$, and $c$.
Section~\ref{sec:skel-chip} shows how to reinterpret the defining conditions
in the language of chip firing. We are led to a natural
generalization to chip firing on a complete graph $K_{1+n}$ in which edges 
between non-sink edges are weighted by $m$, edges touching the sink are
weighted by $c$, and the chip count at each vertex belongs to $\mcG$
rather than $\Z$.

The enumerative results in Theorems~\ref{thm:main-skp} and~\ref{thm:main-skf},
along with the chip-firing interpretation just mentioned, make critical use
of the assumption that chip counts and $x$-coordinates of north steps
come from the additive subgroup $\mcG$. These results fail if we try to
restrict to ``integral'' objects (see Example~\ref{ex:dep}). When
$\mcG=\Z$ and $c,m\in\Z$, we obtain finite collections of $k$-skeletal
paths and functions that are counted (respectively) by $m$-ballot numbers
and generalized parking function counts. However, when $\mcG$
is any non-cyclic subgroup of $\R$, we obtain infinite collections
of $k$-skeletal objects. For coprime positive integers $a,b$, the
\emph{rational Catalan number} $\frac{1}{a+b}\binom{a+b}{a,b}$ counts
rational-slope Dyck paths contained in a triangle bounded by $x=(a/b)y$.
It is natural to ask if there is some alternate version of our $k$-skeletal
constructions that leads to new collections of paths that are counted
(for all $k$) by the rational Catalan number. We address this problem
in a forthcoming paper~\cite{bwl-forthcoming}, 
which requires yet another novel variation of the classical chip-firing model.

\subsection{Motivation from Algebraic and Geometric Aspects of 
Chip Firing} 
\label{subsec:motivation}

The simple definition of chip firing on graphs belies the richness of
this theory and its myriad connections with other areas of
mathematics.  These areas include statistical physics~\cite{bak,
  Dhar}; arithmetic
geometry~\cite{raynaud1970specialisation,Lorenzini89,Lorenzini91};
poset theory~\cite{pretzel1986reorienting, mosesian1972strongly,
  propp2002lattice}; lattice
theory~\cite{amini2010riemann,bacher1997lattice,biggs}; tropical
geometry~\cite{baker-norine,mikhalkin-zharkov}; algebraic
combinatorics~\cite{cori2016hall,Sandpile2024} and commutative
algebra~\cite{rossin, postnikov-shapiro, perkinson}. In
this paper we focus on leveraging connections to the combinatorics of
trapezoidal lattice paths, both unlabeled and labeled.

The chip-firing process is not limited to complete graphs.
We may start with any connected, undirected graph $G =(V,E)$ with
vertex set $V=\{v_1,\ldots,v_N\}$ and edge set
$E=\{e_1,\ldots,e_M\}$. We assume $G$ has no loops, but $G$ may have
multiple edges between two vertices. We choose a
distinguished \emph{sink} vertex $q \in V$ and set
$V'=V\setminus\{q\}$. As before, a \emph{chip configuration} $\ch$ is an
assignment of an integer number of chips to each vertex of $G$.  The
\emph{degree} of a chip configuration $\ch$ is $\sum_{v\in
  V}\ch(v)$. Given a chip configuration $\ch$ of known degree $d$, we
have $\ch(q)=d-\sum_{v\in V'} \ch(v)$. So when focusing on chip
configurations of a fixed degree, we may safely ignore the chip count
$\ch(q)$ at the sink.

The action of firing a vertex or a set of vertices can be expressed in
terms of the \emph{Laplacian of $G$}, the $N\times N$ matrix $\LL$
with entries
\begin{equation*}
\LL_{i,j}=\begin{cases}
          \,\,\text{degree of $v_i$}, \quad & \text{if}\,\, i=j;  \\
          \,\, -\text{(number of edges linking $v_i$ and $v_j$)},
     \quad & \text{if}\,\,  i\neq j. \end{cases}
\end{equation*}
For $S\subseteq V$, let $\ee_S$ be the column vector with $1$s in
those positions $i$ for which $v_i\in S$ and $0$s elsewhere. The
configuration $\ch'=\ch-\LL\ee_S$ is defined as the configuration
obtained from $\ch$ by \emph{cluster firing} the set $S$.

In the study of chip firing, we are often interested in certain
distinguished chip configurations. As sketched in
Section~\ref{subsec:kskel} for the sandpile model on $G=K_{1+n}$, 
these distinguished configurations include the \emph{critical} (also called
\emph{recurrent}) configurations and the
\emph{superstable} configurations. In terms of the Laplacian, the
superstable chip configurations $\ch$ are those satisfying:
\begin{enumerate}
    \item $\ch(v)\geq 0$ for all $v\in V'$; and 
    \item for every $\emptyset \neq S \subseteq V'$, 
      there exists $v \in V'$ such that $(\ch-\LL\ee_S)(v)<0$.
\end{enumerate}
The first condition means that no non-sink
vertex has a negative number of chips. The second condition means
that there do not exist any non-trivial legal cluster firings. These
superstable configurations are equivalent to the \emph{$G$-parking
functions} of Postnikov \& Shapiro~\cite{postnikov-shapiro} and the
\emph{$q$-reduced divisors} of Baker \& Norine~\cite{baker-norine}. The
duality with critical configurations is closely related to
Riemann--Roch duality for graphs~\cite{baker-norine} and Alexander duality for monomial
ideals~\cite{manjunath2013monomials}.

Parking functions of order $n$ are in bijection with trees on $n+1$
(labeled) vertices~\cite{kreweras}. It is a classical result
that the number of parking functions of order $n$ is $(n+1)^{n-1}$
(due to Cayley~\cite{cayley} in the context of labeled trees; see
also~\cite{konheim,riordan}).

\begin{prop}[see \cite{gaydarov2016parking, kostic2008multiparking, postnikov-shapiro}]
  If $G$ is the complete graph on $n+1$ vertices, then the set of
  $G$-parking functions (and hence the set of superstable
  configurations and the set of $q$-reduced divisors of degree $0$) can be naturally
  identified with the set of parking functions of order $n$.
\end{prop}
This connection holds more generally. The number of
$G$-parking functions is counted by the number of spanning
trees of $G$~\cite{postnikov-shapiro}. To explain this connection, 
it is helpful to describe how the critical configurations arise
algebraically.

We say that chip configurations $\ch$ and $\ch^*$ are \emph{chip-firing
equivalent}, written $\ch \sim \ch^*$, if we can obtain $\ch^*$ from $\ch$ by
a finite sequence of chip-firing moves.  We have $\ch\sim \ch^*$ if and
only if there is an integer vector $\ww\in\Z^N$ with $\ch-\ch^*=\LL\ww$.
Here we are not concerned with whether vertices have
negative numbers of chips, so we do not need to distinguish between
the cluster fire of a set $S$ and the sequential firing of all
vertices in $S$; we can stay within the abelian sandpile model.

The set of chip configurations on $G$ is an additive group isomorphic
to $\Z^N$ if we add chip configurations pointwise.
The \emph{chip-firing equivalence class} of $\ch$ is 
$[\ch]=\{\ch^*: \ch^*\sim \ch\}$. Equivalent configurations have the same
degree, so it makes sense to define $\deg([\ch])=\deg(\ch)$.
The collection of chip-firing equivalence classes forms an additive group
$\Pic(G)=\Z^N/\LL(\Z^N)$ called the \emph{Picard group}. 
Let $\Pic^d(G)$ be the subset of $\Pic(G)$ consisting of
classes $[\ch]$ where $\deg(\ch)=d$. The set $\Pic^0(G)$
is a subgroup of $\Pic(G)$, called the \emph{critical group} (this
group goes by different names depending on the context in which
it is introduced: the \emph{chip-firing group}, the \emph{Jacobian},
and the \emph{sandpile group}).

For a connected graph $G$, there is a natural isomorphism $\Pic(G)
\cong \Z\oplus \Pic^0(G)$.  So $\Pic^0(G)$ can be identified with the
torsion part of the $\Z$-cokernel of $\LL$.  The critical
configurations central to this paper can be taken to be
representatives of elements of $\Pic^0(G)$.  We can take the number of
chips on the sink to be the negative of the total number on the
non-sink vertices, leading to a total degree of zero.  As $G$ is
connected, it can be shown by elementary group theory that for any
$i$, $\left|\Pic^0(G)\right| = \left|\det(\LL_i)\right|$, where
$\LL_i$ is the matrix obtained by deleting $i$\nobreakdash-th row and
column of $\LL$.  On the other hand, Kirchhoff's Matrix-Tree
Theorem~\cite[\S3.18]{loehr-comb} states that
$\left|\det(\LL_i)\right|$ is the number of spanning trees of $G$,
hence $\left|\Pic^0(G)\right|$ is equal to the number of spanning
trees of a graph (see~\cite{baker2013chip}).  Several explicit
bijections between $\Pic^0(G)$ and the set of spanning trees of $G$
are known~(see~\cite{an2014canonical, Backman-bij, benson,
  cori2003sand, Dhar, kostic2008multiparking, perkinson2017g, yuen}).

We have outlined some of the close connections among critical
configurations, $G$-parking functions, $\Pic^0(G)$ and spanning trees
of $G$. Statistical physicists Caracciolo, Paoletti, and
Sportiello~\cite{caracciolo2012multiple} and the first
author~\cite{Backman-bij} independently discovered a generalization of
$G$-parking functions with respect to an abstract simplicial
complex.\footnote{The first author thanks Lionel Levine for first
observing the connection between these works.} 
Recall that a simplicial complex $\Delta$ on $V'$ is a collection of
nonempty subsets of $V'$ such that $\{v\}\in\Delta$ for all $v\in V'$,
and whenever $\emptyset\neq T\subseteq S\in\Delta$, $T$ also belongs
to $\Delta$. We allow the vertices of any subset in $\Delta$ to fire
simultaneously. Regarding the objects introduced
in~\cite{Backman-bij}, a chip configuration $\ch:V\rightarrow\Z$ will
be termed \emph{$\Delta$-critical} if these three conditions hold:
\begin{enumerate}
\item \label{itm:oneG} For each $v\in V', \ch(v)\geq 0$.
\item \label{itm:twoG} For each $S \in \Delta$, there exists $v\in S$
  with $(\ch-\LL\ee_S)(v)<0$.
\item For each nonempty $S\subseteq V'$, the configuration
  $\ch+\LL\ee_S$ does not satisfy both Conditions~(\ref{itm:oneG})
  and~(\ref{itm:twoG}).
\end{enumerate}
In the case where the complex $\Delta$ consists of all one-element
subsets of $V'$, we recover the sandpile model; the
$\Delta$-critical configurations are the critical 
(recurrent) configurations. At the other extreme, when $\Delta$ is
the full complex consisting of all nonempty subsets of $V'$,
we recover the unconstrained chip-firing model; the
$\Delta$-critical configurations are the superstable
configurations.

The first author was motivated to consider such simplicial complexes
in the context of chip firing by a desire to develop a notion of
divisor theory for tropical curves with respect to an open cover as
well as a discrete version for graphs. For a fixed graph $G$, sink $q$,
and simplicial complex $\Delta$, the $\Delta$-critical configurations
interpolate between the recurrent configurations and 
the superstable configurations. All such interpolations are
equinumerous with the set of spanning trees in the graph.

\begin{theorem}[Theorem 1 \& Lemma 2 \cite{Backman-bij}, Section 3 \cite{caracciolo2012multiple}] 
\label{thm:delta-critical}
  For each fixed sink vertex $q$ and simplicial complex $\Delta$ on $V'$,
  every chip configuration $\ch$ is equivalent to a unique $\Delta$-critical 
  configuration.
\end{theorem}

\begin{corollary}
  The number of $\Delta$-critical configurations is the number of 
  spanning trees of $G$.
\end{corollary}

An explicit bijection between $\Delta$-critical configurations 
and spanning trees was provided by the first author in~\cite{Backman-bij}.

The motivating idea for the present work was to specialize this
construction to $G=K_{1+n}$ and take $\Delta$ to be the $k$-skeleton of
the full complex (so $\Delta$ consists of all subsets of $\{1,2,\ldots,n\}$ of
size between $1$ and $k+1$). The advantage of this setup is that $\Delta$ is
invariant under $S_n$, so we are able to investigate not only the
corresponding generalization of classical parking functions, but also
the analogues of Dyck paths (the $S_n$-orbits of the parking
functions), which connect to combinatorial representation theory and
adjacent fields. These objects admit an intrinsic description without
any reference to chip firing. As described in
Theorem~\ref{thm:main-skp} and Theorem~\ref{thm:main-skf}, our
construction extends to generalizations of rational parking functions
and trapezoidal lattice paths that go beyond the framework previously studied 
by the first author. We are hopeful that these objects will find
applications in the study of $q,t$-Catalan combinatorics, where parking
functions and Dyck paths have previously been utilized. 
Section~\ref{sec:t-analog} explains how the $q,t$-Catalan $\dinv$ statistics 
interact nicely with our $k$-skeletal constructions, allowing us to define
$k$-skeletal versions of the $q=1$ specialization of the $q,t$-Catalan
that are independent of $k$.

\section{Skeletal Paths}
\label{sec:skel-paths}

\subsection{Definitions and Main Result}
\label{subsec:path-def}

We first introduce certain paths that generalize lattice paths.
Throughout, $n$ is a fixed positive integer and $\mcG$ is an
additive subgroup of $\R$ (typically $\mcG=\Z$ or $\mcG=\R$).
Formally, we define a \emph{path of height $n$ with values in $\mcG$} 
to be a set $\pi=\{(x_i,i): i=0,1,2,\ldots,n-1\}$ of $n$ points in $\R^2$
such that $x_0\leq x_1\leq \cdots\leq x_{n-1}$ and all $x_i\in\mcG$. Informally,
we make the \emph{picture of the path $\pi$} by drawing $n$ unit-length north 
steps from $(x_i,i)$ to $(x_i,i+1)$ for $0\leq i\leq n-1$, and drawing
east steps connecting $(x_i,i+1)$ to $(x_{i+1},i+1)$ for $0\leq i<n-1$.
The points $(x_i,i)$ are called the \emph{vertices} of the path $\pi$.
The formal definition focuses on these vertices (the starting points
of the north steps) to make later connections to parking functions
and chip-firing configurations more transparent and to avoid ambiguities
involving initial east steps and final east steps.

We say the path $\pi=\{(x_0,0),\ldots,(x_{n-1},n-1)\}$ 
is \emph{nonnegative} if $x_0\geq 0$.
For such a path, we often draw an initial east step from $(0,0)$ 
to $(x_0,0)$.  For sets of paths where there is a known upper bound $M$ for
all $x_i$, we may also draw a final east step from $(x_{n-1},n)$ to $(M,n)$.
When $\mcG=\Z$, nonnegative paths can be identified with classical 
\emph{lattice paths}, which are sequences of unit-length north steps and 
unit-length east steps starting at the origin. We write $P_n=P_n(\mcG)$ 
for the set of all paths of height $n$ with values in $\mcG$.

Fix parameters $c,m\in\mcG$ with $c>0$ and $m\geq 0$.
The \emph{reference line} for these parameters is the line with equation
$x=my+c$.  For each integer $k\in\{0,1,\ldots,n-1\}$, a path
$\pi\in P_n$ is called a \emph{k-skeletal path (for parameters $c$ and $m$)} 
if and only if these conditions hold:
\begin{itemize}
\item[(P0)] $\pi$ is nonnegative.
\item[(P1)] The last $k+1$ north steps of $\pi$ all start strictly left
of the line $x=my+c$.
\item[(P2)] There do not exist $k+1$ consecutive rows
such that the north steps of $\pi$ in those rows
all start strictly left of the line $x=my$.
\end{itemize}
In terms of the vertices $(x_i,i)$ of $\pi$, 
Condition~(P0) requires that $x_0\geq 0$.  
Condition~(P1) requires that $x_i<mi+c$ for all $i$ 
in $\{n-1,n-2,\ldots,n-(k+1)\}$. In the special case $\mcG=\Z$,
we could equivalently require $x_i\leq mi+c-1$, but the general
case requires the strict inequality.
Condition~(P2) requires that for any $i$ in $\{0,1,\ldots,n-(k+1)\}$, 
there must exist $j$ in $\{i,i+1,i+2\ldots,i+k\}$ with $x_j\geq mj$.
Let $\skp_k = \skp_k(\mcG,n;c,m)$ be the set of $k$-skeletal paths of
height $n$ with parameters $c$ and $m$ and values in $\mcG$. The main 
theorem of this section is the following:

\begin{theorem}\label{thm:main-skp}
For all $k,k'\in\{0,1,\ldots,n-1\}$, there is a canonical bijection
from $\skp_k$ to $\skp_{k'}$. 
\end{theorem}

The proof of Theorem~\ref{thm:main-skp} appears in
Sections~\ref{subsec:area-vec} through~\ref{subsec:analyze-conditions}. 
In Section~\ref{subsec:area-vec} we introduce the notion of an area
vector $g$ associated with a path and rephrase Conditions~(P0),~(P1), and~(P2) 
in terms of area vectors. In Section~\ref{subsec:cycle-area-vec} we introduce
certain equivalence classes of area vectors that, by
Theorem~\ref{thm:skp-eqv-class} of that section, contain exactly one
$k$-skeletal area vector for each $k$ between $0$ and $n-1$.
The bijection of Theorem~\ref{thm:main-skp} thereby arises by mapping any
given $k$-skeletal area vector to the corresponding $k'$-skeletal area
vector in its equivalence class. Thus, the proof of
Theorem~\ref{thm:main-skp} is reduced to that of
Theorem~\ref{thm:skp-eqv-class}, which is carried out in
Sections~\ref{subsec:prelim-lemma} and~\ref{subsec:analyze-conditions}.

Before continuing on to the proof of Theorem~\ref{thm:main-skp}, we
consider its enumerative consequences for the case of $\mcG=\Z$.

\begin{corollary}\label{cor:main-enum}
  When $\mcG=\Z$, the number of $k$-skeletal paths is given by an
  $m$-ballot number: for all $k$ between $0$ and $n-1$, 
  \begin{equation}\label{eq:main-enum}
    |\skp_k|=\frac{c}{(m+1)n+c}\binom{(m+1)n+c}{n}. 
  \end{equation}
\end{corollary}
\begin{proof}
  For $k=n-1$, equation~\eqref{eq:main-enum} is a classical result.
  Condition~(P1) restricts attention to paths whose north steps start
  at points $(x,y)\in\Z^2$ with $x<my+c$, or equivalently $x\leq my+c-1$. 
 Condition~(P2) is
 automatically satisfied since the first north step of any such path
  starts at $(x_0,0)$ with $x_0\geq 0$.  We can identify these paths
  with lattice paths from $(0,0)$ to $(mn+c-1,n)$ contained in the
  trapezoidal region $\mcT=\{(x,y)\in\R^2: 0\leq y\leq n,\ 0\leq x\leq
  my+c-1\}$.  The stated formula for the number of such lattice paths
  can be shown by induction 
(see, for example,~\cite[Theorem~2.27]{loehr-comb}).

  For $k < n-1$, the claimed enumeration follows from the case of
  $k=n-1$ along with the bijection from Theorem~\ref{thm:main-skp}.
\end{proof}

\begin{remark}\label{rmk:rotate}
  When $k=0$, we can interpret $\skp_0$ as counting lattice paths in a
  rotated version of the trapezoid $\mcT$ referenced in the proof
  of Corollary~\ref{cor:main-enum}. By adjusting east steps we obtain
  a lattice path from $(-m,0)$ to $(mn-m+c-1,n)$. Under this
  bijection, $\skp_0$ is the set of lattice paths contained in the
  region $\{(x,y)\in\R^2: 0\leq y\leq n,\ my-m\leq x\leq
  mn-m+c-1\}$. This is $\mcT$ rotated 180 degrees about the point
  $\frac{1}{2}(mn-m+c-1,n)$. Figure~\ref{fig:rotation} gives an example
 where $n=2$, $c=1$, and $m=2$.
\end{remark}

\begin{example}\label{ex:skp3}
Let $\mcG=\Z$, $n=3$, and $c=m=1$. The reference line is $x=y+1$.
Here, we can view skeletal paths as lattice paths from $(0,0)$ to $(3,3)$
with three unit-length north steps and three unit-length east steps. We find
\begin{align*}
  \skp_0&=\{\N\E\N\E\N\E,\E\N\E\N\N\E,\N\E\E\N\N\E,\E\N\N\E\N\E,\E\E\N\N\N\E\},
\\\skp_1&=\{\N\E\N\E\N\E,\N\E\N\N\E\E,\N\N\E\E\N\E,\E\N\N\E\N\E,\E\N\N\N\E\E\},
\\\skp_2&=\{\N\E\N\E\N\E,\N\E\N\N\E\E,\N\N\E\E\N\E,\N\N\E\N\E\E,\N\N\N\E\E\E\},
\end{align*}
as illustrated in Figure~\ref{fig:ex3}. The set $\skp_2$ consists of
Dyck paths in the triangle bounded by $x=0$, $y=3$, and $x=y$. 
Note these paths can touch the line $x=y$ but are strictly left
of the reference line $x=y+1$.
Adjusting east steps in $\skp_0$ to get paths from $(-1,0)$ to
$(2,3)$ as described in Remark~\ref{rmk:rotate}, we get
\[ \{\E\N\E\N\E\N,\E\N\E\E\N\N,\E\E\E\N\N\N,\E\E\N\E\N\N,\E\E\N\N\E\N\},
\] which is the set of rotated Dyck paths in the triangle bounded by
$x=y-1$, $y=0$, and $x=2$.
\end{example}
\begin{figure}[h]
\begin{center}
  \includegraphics[width=0.9\linewidth]{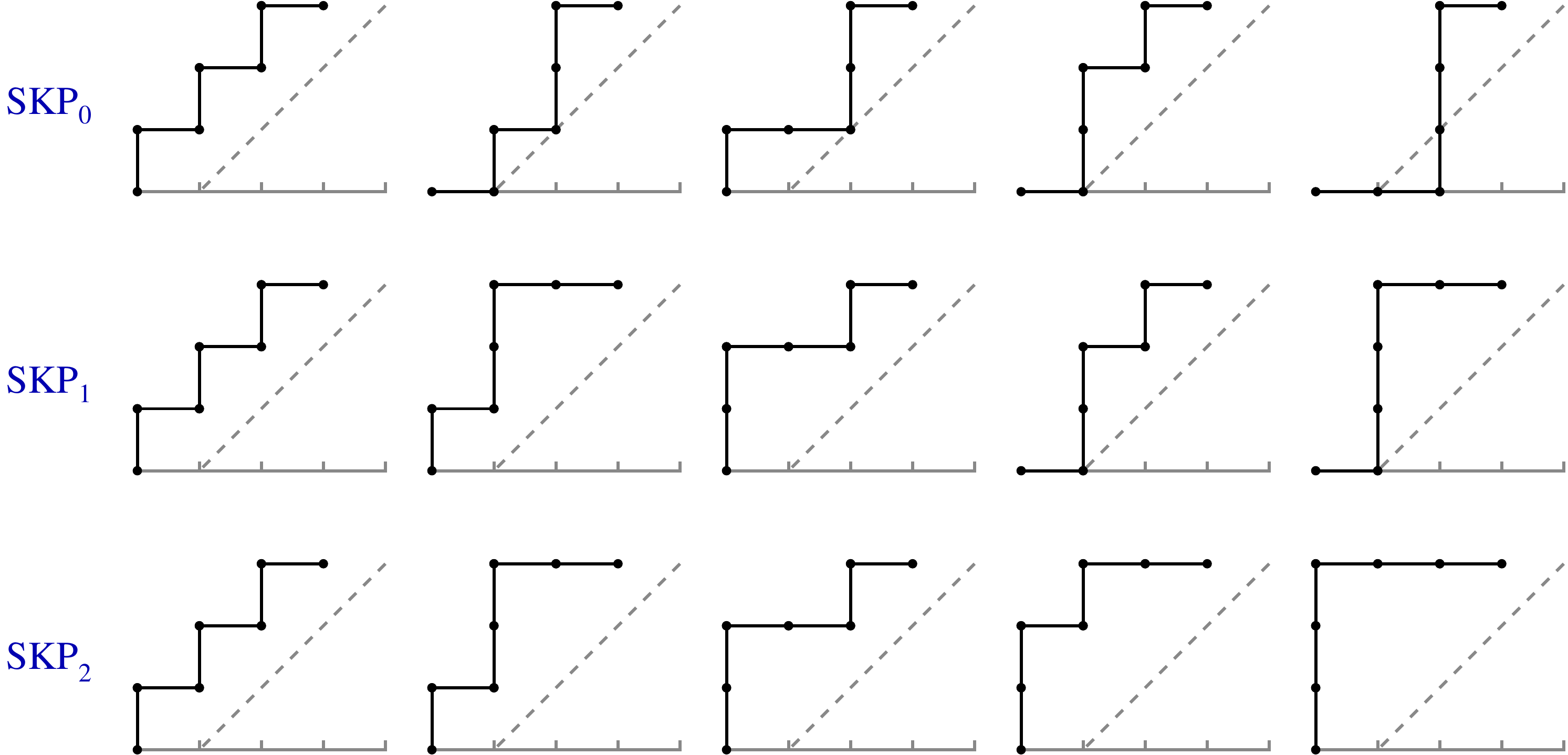}
  \caption{The three collections of paths from Example~\ref{ex:skp3}.}
\label{fig:ex3}
\end{center}
\end{figure}

When $\mcG$ is a non-cyclic subgroup of $\R$, 
the collections of $k$-skeletal paths are infinite,
but Theorem~\ref{thm:main-skp} still applies.
We could obtain finite subcollections by imposing the extra
requirement that all north steps of paths have integer $x$-coordinates.
However, Theorem~\ref{thm:main-skp} no longer holds in this setting,
as seen in the next example.

\begin{figure}[h]
\begin{center}
  \includegraphics[width=0.9\linewidth]{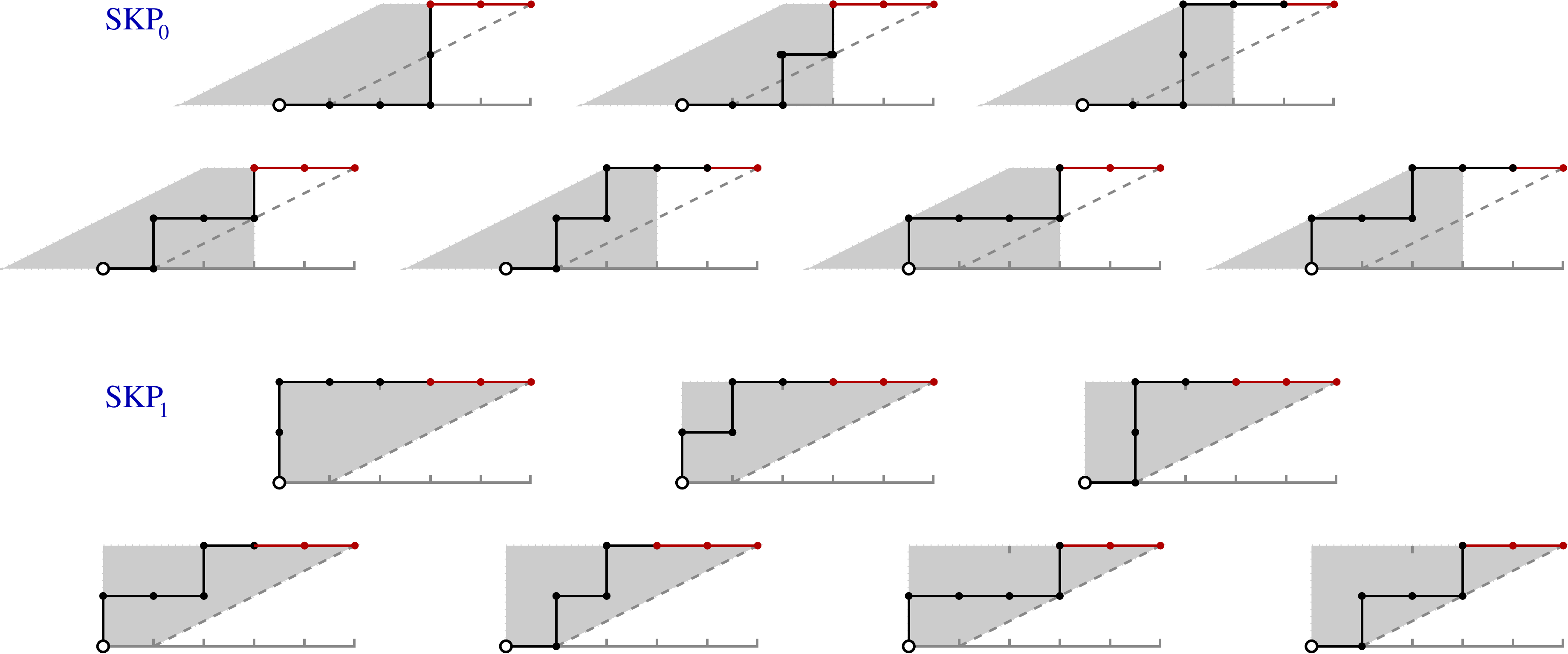}
  \caption{Illustration of the duality between $\skp_0$ and
    $\skp_{n-1}$ in the case of $n=2$, $m=2$, and $c=1$ as described in
    Remark~\ref{rmk:rotate}. The black portion of each path in
    $\skp_0$ can be found, after a 180-degree rotation, as the black
    portion of the corresponding path in $\skp_1$.}
\label{fig:rotation}
\end{center}
\end{figure}

\begin{example}\label{ex:dep}
  Let $\mcG=\Q$, $n=2$, $m=3/2$ and $c=1/2$. 
 When $k=0$, there are three lattice paths
  (consisting of unit-length east steps and north steps) that
  satisfy Condition~(P1), namely $\E\N\N\E\E$, $\N\E\N\E\E$, and $\N\N\E\E\E$. 
  None of these paths satisfy Condition~(P2). But for $k=1$,
  the two lattice paths $\N\E\N\E\E$ and $\N\N\E\E\E$ satisfy both conditions.
 Compare to Example~\ref{ex:dep2} below, which finds the (non-integral)
   $0$-skeletal objects corresponding to these two paths.
\end{example}

\subsection{Area Vectors}
\label{subsec:area-vec}

To study skeletal paths, we develop a bijection between paths
of height $n$ and the area vectors defined next.
An \emph{area vector} is an $n$-tuple 
$g=(g_0,g_1,\ldots,g_{n-1})\in\mcG^n$ such that
$g_{i+1}\leq g_i+m$ for $0\leq i<n-1$. Let $\AV_n$ be the set of all
such area vectors ($\AV_n$ also depends on $m$ and $\mcG$).
We call $g\in\AV_n$ a \emph{Dyck} vector iff all $g_i>0$.

For a path $\pi=\{(x_i,i)\}\in P_n$, 
define the \emph{area vector of $\pi$} to be 
\[ G(\pi)=(g_0,g_1,\ldots,g_{n-1}),\mbox{ where } g_i=mi+c-x_i
\mbox{ for }0\leq i<n. \]
Each $g_i$ is the signed horizontal distance from vertex $(x_i,i)$
to the reference line $x=my+c$. Since $m,c$ belong to the subgroup $\mcG$,
$x_i$ is in $\mcG$ iff $g_i$ is in $\mcG$.
The inequality $x_i\leq x_{i+1}$ is equivalent to $g_{i+1}\leq g_i+m$.  
It follows that $G:P_n\rightarrow\AV_n$ is a bijection.
From the definition of $g_i$, we also deduce:
$x_0\geq 0$ iff $g_0\leq c$; $x_i<mi+c$ iff $g_i>0$; $x_i<mi$ iff $g_i>c$.
Thus, we can rephrase the definition of skeletal paths in terms
of area vectors, as follows.

\begin{prop}
Let $g=(g_0,g_1,\ldots,g_{n-1})$ be the area vector of a path $\pi$.
The path $\pi$ is a $k$-skeletal path (for parameters $c$ and $m$)
if and only if these conditions hold:
\begin{itemize}
\item[(A0)] $g_0\leq c$.
\item[(A1)] The last $k+1$ entries $g_{n-1},\ldots,g_{n-(k+1)}$
of $g$ are all strictly positive.
\item[(A2)] There do not exist $k+1$ consecutive entries $g_i,\ldots,g_{i+k}$
in $g$ that all strictly exceed $c$.
\end{itemize}
\end{prop}

We call $g\in\AV_n$ a \emph{$k$-skeletal} area vector iff $g$
satisfies (A0), (A1), and (A2). Let $\skv_k$ be the set of such area vectors.

\begin{example}
Let $G=\Z$, $n=3$, and $c=m=1$. The $k$-skeletal paths found in
Example~\ref{ex:skp3} correspond to the following area vectors:
\begin{align*}
\skv_0&=\{(1,1,1),(0,0,1),(1,0,1),(0,1,1),(-1,0,1)\}, \\
\skv_1&=\{(1,1,1),(1,1,2),(1,2,1),(0,1,1),(0,1,2)\}, \\
\skv_2&=\{(1,1,1),(1,1,2),(1,2,1),(1,2,2),(1,2,3)\}.
\end{align*}
\end{example}

\begin{example}\label{ex:kex}
Let $G=\Z$, $n=14$, $c=4$, and $m=1$. Consider the path $\pi=\{(x_i,i)\}$,
where \[ (x_0,x_1,\ldots,x_{13})=(0,0,0,4,4,4,5,9,13,13,13,13,13,14). \]
Viewing $\pi$ as a lattice path from $(0,0)$ to $(14,14)$
as shown in Figure~\ref{fig:kex}, we have
\[ \pi=\N\N\N\E\E\E\E\N\N\N\E\N\E\E\E\E\N\E\E\E\E\N\N\N\N\N\E\N. \]
The area vector of $\pi$ is $G(\pi)=(4,5,6,3,4,5,5,2,-1,0,1,2,3,3)$. 
Since $G(\pi)$ ends in exactly four positive entries,
Condition~(A1) holds for all $k\leq 3$.
Since $G(\pi)$ contains subsequences $5,6$ and $5,5$ but no longer
consecutive subsequence of entries exceeding $4$,
Condition~(A2) holds for all $k\geq 2$. Thus,
the path $\pi$ and area vector $G(\pi)$ are $k$-skeletal for $k=2,3$ (only).
\end{example}
\begin{figure}[h]
\begin{center}
  \includegraphics[width=0.5\linewidth]{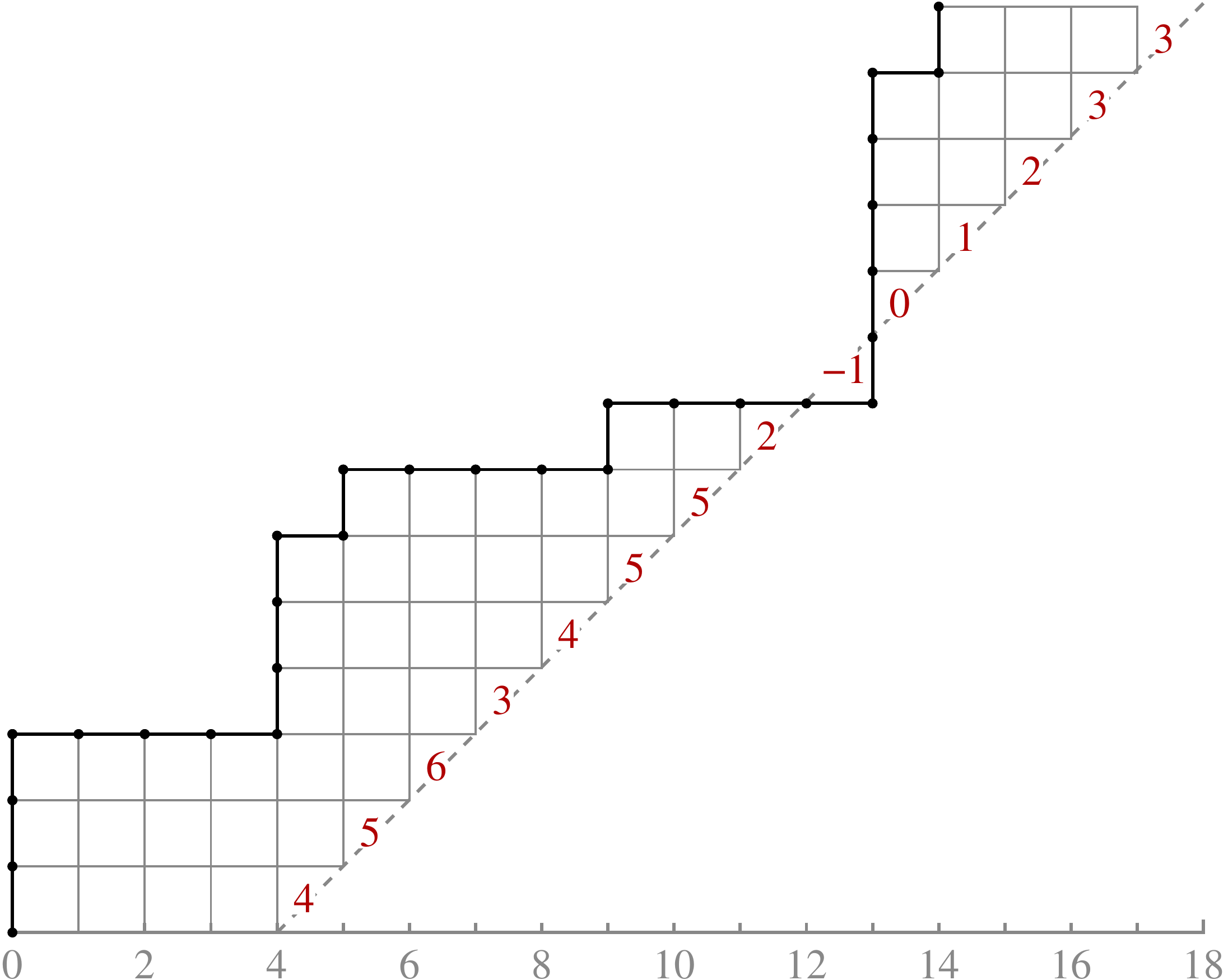}
  \caption{The path $\pi$ from Example~\ref{ex:kex}.}
\label{fig:kex}
\end{center}
\end{figure}

\subsection{Cycling Operator on Area Vectors}
\label{subsec:cycle-area-vec}

Define the \emph{cycling operator} $C:\mcG^n\rightarrow\mcG^n$ by
\[ C(g_0,g_1,\ldots,g_{n-1})=(g_1,\ldots,g_{n-1},g_0-c). \]
$C$ is a bijection with inverse
$C^{-1}(h_0,h_1,\ldots,h_{n-1})=(h_{n-1}+c,h_0,h_1,\ldots,h_{n-2})$. 
We write $C^2=C\circ C$, $C^{-2}=C^{-1}\circ C^{-1}$, and so on.

For $g=(g_0,g_1,\ldots,g_{n-1})\in\mcG^n$, 
define $\area(g)=g_0+g_1+\cdots+g_{n-1}$.
We see that $\area(C(g))=\area(g)-c$.
For $g,h\in\mcG^n$, write $g\succeq h$ to mean $\area(g)\geq\area(h)$.
Define $\pos(g)\in\{0,1,\ldots,n\}$ to be the largest integer $\ell$ 
such that the last $\ell$ entries of $g$ are strictly positive.
Observe that $g$ is a Dyck vector iff $\pos(g)=n$.

Define $S=\{g\in\AV_n: g_0\leq c\mbox{ and }g_{n-1}>0\}$.  Elements of
$S$ are the area vectors of nonnegative paths $\pi$ whose
last vertex is strictly left of the reference line.  For any path
$\pi$ satisfying Conditions~(P0) and (for some $k$)~(P1), 
$G(\pi)$ belongs to $S$.
Define an equivalence relation $\sim$ on $S$ as follows: for $g,h\in
S$, $g\sim h$ means $h=C^j(g)$ for some $j\in\Z$. Informally, this
says that we can go from $g$ to $h$ by applying $C$ or $C^{-1}$
finitely many times, noting that some intermediate vectors along the
way might not belong to $S$.  Since applying $C$ decreases area, each
equivalence class of $\sim$ is totally ordered by the relation
$\succeq$. 

\begin{theorem}\label{thm:skp-eqv-class}
For each $k\in\{0,1,\ldots,n-1\}$, each equivalence class $T$ of
$\sim$ in $S$ contains exactly one $k$-skeletal area vector, namely
the least $g$ in $T$ (relative to $\succeq$) satisfying $\pos(g)>k$.
\end{theorem}

Theorem~\ref{thm:main-skp} follows immediately from
Theorem~\ref{thm:skp-eqv-class}:
\begin{proof}[Proof of Theorem~\ref{thm:main-skp}]
  We define a bijection $\skp_k\rightarrow\skp_{k'}$ by mapping each
  $k$-skeletal path $\pi$ to the unique $k'$-skeletal path $\pi'$
  with $G(\pi')\sim G(\pi)$.
\end{proof}

\begin{example}\label{ex:dep2}
Let $\mcG=\Q$, $n=2$, $m=3/2$ and $c=1/2$.  
Two of the infinitely many $1$-skeletal area vectors for these parameters
are $g=(1/2,2)$ and $g'=(1/2,1)$. These area vectors correspond to
the two lattice paths $\N\N\E\E\E$ and $\N\E\N\E\E$ from Example~\ref{ex:dep}.
Applying $C$ to $g$ six times leads to the $0$-skeletal area vector
$(-1,1/2)$ in the same equivalence class as $g$.
Applying $C$ to $g'$ twice leads to the $0$-skeletal area vector
$(0,1/2)$.
\end{example}

The remaining two subsections contain the proof of
Theorem~\ref{thm:skp-eqv-class}.

\subsection{Preliminary Lemmas}
\label{subsec:prelim-lemma}

\begin{lemma}\label{lem:area-vec}
For all $g=(g_0,g_1,\ldots,g_{n-1})\in S$ and all $j\in\Z$,
$C^j(g)$ belongs to $\AV_n$.
\end{lemma}
\begin{proof}
Write $j=ne+p$, where $e\in\Z$ and $0\leq p<n$. If $p=0$, then
\begin{equation}\label{eq:C-en}
 C^j(g)=C^{ne}(g)=(g_0-ec,g_1-ec,\ldots,g_{n-1}-ec). 
\end{equation}
Because $g\in S$, we know $g_{i+1}\leq g_i+m$ for $0\leq i<n-1$.
Adding $-ec$ to both sides of these inequalities, we see that
$C^j(g)$ is an area vector.
If $0<p<n$, then
\begin{equation}\label{eq:C-en+p}
C^j(g)=C^{ne+p}(g)=(g_p-ec,\ldots,g_{n-1}-ec;g_0-(e+1)c,\ldots,g_{p-1}-(e+1)c). 
\end{equation}
Each pair of consecutive entries separated by a comma satisfies the
needed inequality to be an area vector,  as we see
by adding $-ec$ or $-(e+1)c$ to one of the inequalities
$g_{i+1}\leq g_i+m$. We must also check that the pair of entries separated by
a semicolon satisfies $g_0-(e+1)c\leq g_{n-1}-ec+m$, or equivalently
$g_0-m\leq c+g_{n-1}$. Because $g\in S$, this last inequality follows
by adding the inequalities $g_0\leq c$ and $-m\leq 0<g_{n-1}$. 
We even have the \emph{strict} inequality $g_0-(e+1)c<(g_{n-1}-ec)+m$,
a fact that we will need later.
\end{proof}

\begin{lemma}\label{lem:eqv-fin}
Each equivalence class of $\sim$ in $S$ is finite.
\end{lemma}
\begin{proof}
Consider the equivalence class of some $g\in S$.
Since $c>0$, there exists $j_0\in\Z_{>0}$ so that for all $j\geq j_0$,
all entries in $C^j(g)$ are nonpositive. This is clear from
formulas~\eqref{eq:C-en} and~\eqref{eq:C-en+p}. Similarly,
there exists $i_0\in\Z_{<0}$ so that for all $i\leq i_0$,
all entries in $C^i(g)$ exceed $c$. By definition of $S$,
the equivalence class of $g$ must be a subset of the finite
set $\{C^r(g): i_0<r<j_0\}$.
\end{proof}

The next lemma proves the $k=n-1$ case of Theorem~\ref{thm:skp-eqv-class}.

\begin{lemma}\label{lem:dyck-rep}
Each equivalence class $T$ of $\sim$ contains exactly
one Dyck vector.
\end{lemma}
\begin{proof}
Let $T$ be a fixed equivalence class in $S$. We first show $T$ contains
at most one Dyck vector. Suppose $g=(g_0,g_1,\ldots,g_{n-1})\in T$
is a Dyck vector. Then $g_0\leq c$ (since $g\in S$), so
$C(g)=(g_1,\ldots,g_{n-1},g_0-c)$ has last entry $\leq 0$. 
If we continue to apply $C$, this entry will move left through the vector,
and eventually it cycles back to the right and becomes even smaller.
We see that all vectors following any Dyck vector $g$ in the totally 
ordered set $T$ are not Dyck vectors. Thus $T$ cannot contain two different 
Dyck vectors.

Next we show that $T$ does contain a Dyck vector.
Start with any $g=(g_0,g_1,\ldots,g_{n-1})$ in $T$.
If $g$ is a Dyck vector, then there is nothing to prove.
Otherwise, choose the least integer $e>0$ such that
\[ g^*=C^{-en}(g)=(g_0+ec,g_1+ec,\ldots,g_{n-1}+ec)\in\mcG^n \]
has all positive entries.  Note that $g^*$ may not belong to $S$.
By minimality of $e$, there must exist $j$ with $0<g_j+ec\leq c$.
Choose the least such index $j$, and let
\[ g^+=C^j(g^*)=(g_j+ec,g_{j+1}+ec,\ldots,g_{n-1}+ec,
  g_0+(e-1)c,g_1+(e-1)c,\ldots,g_{j-1}+(e-1)c). \]
In this new vector, the first entry $g_j+ec$ is $\leq c$ by choice of $j$.
All entries (including the last one) are strictly positive, by choice of $e$ 
and $j$.  Also $g^+$ is an area vector by Lemma~\ref{lem:area-vec}.
So $g^+$ is in $S$ and is a Dyck vector in $T$.
\end{proof}

\begin{example}\label{ex:find-pos}
Let $\mcG=\Z$, $n=14$, $c=4$, $m=1$, and 
$g=(3,4,5,5,2,-1,0,1,2,3,3,0,1,2)\in S$.
Following the proof of the lemma, we take $e=1$ to get
$$g^*=C^{-14}(g)=(7,8,9,9,6,3,4,5,6,7,7,4,5,6),$$ which 
has all positive entries but is not in $S$. We find $j=5$ and set
$$g^+=C^5(g^*)=(3,4,5,6,7,7,4,5,6,3,4,5,5,2).$$
This is the unique Dyck vector
in $S$ equivalent to $g$. Starting at $g^+$ and applying $C$ repeatedly,
the equivalence class of $g$ is $\{g^+,g',g,g''\}$, where
\[ g'=C^6(g^+)=(4,5,6,3,4,5,5,2,-1,0,1,2,3,3);\quad
   g=C^9(g^+); \]
\[ g''=C^{13}(g^+)=(2,-1,0,1,2,3,3,0,1,2,-1,0,1,1). \]
We have $\pos(g^+)=14=n$, $\pos(g')=4$, $\pos(g)=2$, and $\pos(g'')=2$.
The definition shows that $g^+$ is $k$-skeletal for $4\leq k<14$,
   $g'$ is $k$-skeletal for $2\leq k<4$,
   $g$ is not $k$-skeletal for any $k$, and
   $g''$ is $k$-skeletal for $0\leq k<2$.
This agrees with the conclusion of Theorem~\ref{thm:skp-eqv-class}.
\end{example}

Given $g\in S$, the next lemma determines which subsequent objects
$C^j(g)$ also belong to $S$.

\begin{lemma}\label{lem:in-S}
Suppose $g=(g_0,g_1,\ldots,g_{n-1})$ is in $S$.
\\ (a) For each $e>0$, $C^{en}(g)\in S$ if and only if 
 $g_0\leq (e+1)c$ and $g_{n-1}>ec$.
\\ (b) For all $e,p$ with $e\geq 0$ and $0<p<n$,
 $C^{en+p}(g)\in S$ if and only if $g_p\leq (e+1)c$ and $g_{p-1}>(e+1)c$.
\end{lemma}
\begin{proof}
Because $g\in S$, Lemma~\ref{lem:area-vec} shows that all vectors $C^j(g)$
are area vectors. Assertions~(a) and~(b) follow at once 
from~\eqref{eq:C-en} and~\eqref{eq:C-en+p} and the definition of $S$.
\end{proof}

\subsection{Analysis of $k$-Skeletal Conditions}
\label{subsec:analyze-conditions}

To finish proving Theorem~\ref{thm:skp-eqv-class}, we reformulate
the $k$-skeletal Conditions~(A0),~(A1), and~(A2) in several ways.

\begin{prop}\label{prop:A012'}
An area vector $g=(g_0,g_1,\ldots,g_{n-1})$ is $k$-skeletal 
iff these conditions hold:
\begin{itemize}
\item[(A0)] $g_0\leq c$.
\item[(A1$'$)] $\pos(g)>k$.
\item[(A2$'$)] For all $p\in\{1,2,\ldots,n\}$, 
if ($p=n$ or $g_p\leq c$) and $g_{p-1}>c$, then 
$\pos(C^p(g))\leq k$.
\end{itemize}
\end{prop}
\begin{proof}
Condition~(A1) says that the last $k+1$ entries of $g$ (and perhaps more
entries) are strictly positive, which is equivalent to $\pos(g)>k$, 
as stated in (A1$'$). 
Next, assume $g\in\AV_n$ fails Condition~(A2).
Choose $i$ and $k$ so that $g_i,g_{i+1},\ldots,g_{i+k}$ all exceed $c$
and either $i+k=n-1$ or $g_{i+k+1}\leq c$. 
In the case $i+k=n-1$, let $p=n$. Then $g_{p-1}=g_{n-1}>c$ and
\[ C^p(g)=(g_0-c,g_1-c,\ldots,g_{n-1}-c) \]
has $\pos(C^p(g))>k$, which means that Condition~(A2$'$) fails.
In the case $i+k<n-1$ and $g_{i+k+1}\leq c$, let $p=i+k+1$.
Then $g_p\leq c$, $g_{p-1}>c$, and
\[ C^p(g)=(g_p,\ldots,g_{n-1},g_0-c,g_1-c,\ldots,g_{p-1}-c) \] 
has $\pos(C^p(g))>k$, so Condition~(A2$'$) fails.
Similarly, reversing the reasoning in the two cases shows that
the failure of (A2$'$) implies the failure of (A2).
So Conditions~(A2) and~(A2$'$) are logically equivalent.
\end{proof}

\begin{prop}
An area vector $g=(g_0,g_1,\ldots,g_{n-1})$ is $k$-skeletal 
iff these conditions hold:
\begin{itemize}
\item[(A0)] $g_0\leq c$.
\item[(A1$'$)] $\pos(g)>k$.
\item[(A2$''$)] For all $p\in\{1,2,\ldots,n\}$ such that
 $C^p(g)\in S$, $\pos(C^p(g))\leq k$.
\end{itemize}
\end{prop}
\begin{proof}
Assume $g\in\AV_n$ satisfies (A0) and~(A1$'$), so $g\in S$.
We need only confirm that the hypothesis
``if ($p=n$ or $g_p\leq c$) and $g_{p-1}>c$'' in (A2$'$)
is logically equivalent to the condition $C^p(g)\in S$.
Consider the case $p=n$.
The hypothesis simplifies to $g_{n-1}>c$. 
Taking $e=1$ in Lemma~\ref{lem:in-S}(a),
$C^n(g)\in S$ iff $g_0\leq 2c$ and $g_{n-1}>c$. The condition $g_0\leq 2c$
is already guaranteed because $g_0\leq c<2c$ (as $g\in S$ and $c>0$).
So $C^n(g)\in S$ is equivalent to $g_{n-1}>c$, as needed.
Consider the case $0<p<n$. The hypothesis in (A2$'$) simplifies to
``$g_p\leq c$ and $g_{p-1}>c$.'' 
Taking $e=0$ in Lemma~\ref{lem:in-S}(b), we see
this condition is equivalent to $C^p(g)\in S$.
\end{proof}

After one last technical adjustment to Condition~(A2), 
we will be ready to prove Theorem~\ref{thm:skp-eqv-class}.

\begin{prop}
An area vector $g=(g_0,g_1,\ldots,g_{n-1})$ is $k$-skeletal 
iff these conditions hold:
\begin{itemize}
\item[(A0)] $g_0\leq c$.
\item[(A1$'$)] $\pos(g)>k$.
\item[(A2$'''$)] For all $p\in\Z_{>0}$ such that
 $C^p(g)\in S$, $\pos(C^p(g))\leq k$.
\end{itemize}
\end{prop}
\begin{proof}
Assume $g\in\AV_n$ satisfies~(A0) and~(A1$'$), so $g\in S$.
Certainly, if $g$ satisfies (A2$'''$), then $g$ satisfies (A2$''$).
To prove the converse, assume $g$ fails (A2$'''$),
meaning there exists $p^*>0$ with $C^{p^*}(g)\in S$ and $\pos(C^{p^*}(g))>k$.
We prove $g$ fails (A2$''$) by finding $p\in\{1,2,\ldots,n\}$ with
$C^p(g)\in S$ and $\pos(C^p(g))>k$.

Case~1: $p^*\leq n$. Then we take $p=p^*$.

Case~2: $p^*=en$ for some $e>1$. Here,
$C^{p^*}(g)=(g_0-ec,g_1-ec,\ldots,g_{n-1}-ec)$. Take $p=n$, so
$C^p(g)=(g_0-c,g_1-c,\ldots,g_{n-1}-c)$. We obtain $C^p(g)$
from $C^{p^*}(g)$ by increasing each entry by $(e-1)c>0$.
Since $\pos(C^{p^*}(g))>k$, the last $k+1$ entries of $C^{p^*}(g)$ 
are positive. So the last $k+1$ entries of $C^p(g)$ are positive,
and $\pos(C^p(g))>k$. Since $C^{p^*}(g)$ is in $S$,
Lemma~\ref{lem:in-S}(a) gives $g_0\leq (e+1)c$ and $g_{n-1}>ec$.
Since $g_0\leq c$ also, we deduce $g_0\leq 2c$ and $g_{n-1}>c$.
Lemma~\ref{lem:in-S}(a) now shows $C^p(g)\in S$, as needed.

Case~3: $p^*=ne+r$ for some $e\geq 1$ and $0<r<n$. By assumption,
\[ C^{p^*}(g)=(g_r-ec,\ldots,g_{n-1}-ec,g_0-(e+1)c,\ldots,g_{r-1}-(e+1)c) \]
is in $S$ and its last $k+1$ entries are positive. Applying $C^{-ne}$
to $C^{p^*}(g)$ increases all entries by $ec$, leading to the vector
\[ C^r(g)=(g_r,\ldots,g_{n-1},g_0-c,\ldots,g_{r-1}-c). \]
This vector also has its last $k+1$ entries positive, but $C^r(g)$
need not be in $S$ because $g_r>c$ could occur.  Let $\ell$ be as
large as possible such that the first $\ell$ entries of $C^r(g)$
strictly exceed $c$.  We have $0\leq \ell \leq n-r$ since the
$(n-r+1)$\nobreakdash-th entry is $g_0-c$, and $g_0-c\leq 0<c$. Let $p=r+\ell$, which
is in $\{1,2,\ldots,n\}$.  We compute
\[ C^p(g)=C^{\ell}(C^r(g))
=(g_p,\ldots,g_{n-1},g_0-c,\ldots,g_{r-1}-c,g_r-c,\ldots,g_{p-1}-c).\]
The $\ell$ values that cycle to the right end when we pass from
$C^r(g)$ to $C^p(g)$ start larger than $c$ and get decreased by $c$,
so the cycled values are still strictly positive. 
So $\pos(C^p(g))\geq \pos(C^r(g))>k$. To finish, we show $C^p(g)\in S$.
If $r<p<n$, then $C^p(g)\in S$ because $g_p\leq c$ (by choice of $\ell$)
and $g_{p-1}-c>0$ (due to $g_{p-1}$ being cycled to the right end).
If $p=r$, then $\ell=0$, so $g_r\leq c$ and $C^p(g)=C^r(g)$ is in $S$.
If $p=n$, then $C^p(g)\in S$ because $g_{n-1}-c>0$ (due to $g_{n-1}$ being
cycled to the right end) and $g_0-c<g_0\leq c$.
\end{proof}

\noindent\emph{Proof of Theorem~\ref{thm:skp-eqv-class}:}
Fix $k\in\{0,1,\ldots,n-1\}$ and an equivalence class $T$ of $\sim$.
Every $g\in T$ belongs to $S$ and therefore satisfies $g_0\leq c$,
as required by (A0). Since applying $C$ decreases area, $g$ 
satisfies~(A1$'$) and~(A2$'''$) iff $g$ is the least object in $T$
(relative to $\succeq$) such that $\pos(g)>k$. 
Because $T$ is finite (Lemma~\ref{lem:eqv-fin})
and contains a Dyck object $g^+$ with $\pos(g^+)=n>k$
(Lemma~\ref{lem:dyck-rep}), there exists a unique least $g\in T$ 
with $\pos(g)>k$. This $g$ is the unique
$k$-skeletal area vector in $T$. 

\section{Skeletal Functions}
\label{sec:skel-fns}

\subsection{Representing Functions as Labeled Paths}
\label{subsec:rep-fn-path}

We continue to assume $\mcG$ is an additive subgroup of $\R$
and $n$ is a fixed positive integer. Let $[n]=\{1,2,\ldots,n\}$,
and let $F_n=F_n(\mcG)$ be the set of all functions $f:[n]\rightarrow\mcG$.
For each such function $f$, we define the (unlabeled) \emph{path of $f$}
to be $\pi(f)=\{(x_i,i):0\leq i<n\}$, where $x_0,x_1,\ldots,x_{n-1}$
is the list of function values $f(1),f(2),\ldots,f(n)$ sorted into
weakly increasing order.  The \emph{labeled path of $f$} is the ordered pair 
$(\pi(f),w)$, where $w=(w_0,w_1,\ldots,w_{n-1})$ is the unique rearrangement
of $1,2,\ldots,n$ such that $f(w_i)=x_i$ for all $i$, and whenever
$x_i=x_{i+1}$, we have $w_i<w_{i+1}$. We call $w$ the \emph{label sequence}
for $f$.  Informally, we draw the labeled path for $f$ as follows.
Put a north step labeled $a$ on the line $x=b$ whenever $f(a)=b$.
Arrange these north steps at different heights so that the $x$-coordinates
weakly increase as we move up the figure, producing a path of height $n$.
For north steps on the same line $x=b$, make their
labels increase reading from bottom to top.  This process defines a bijection
from $F_n(\mcG)$ to the set of labeled paths where all $x$-coordinates
belong to $\mcG$. The inverse bijection maps a labeled path to the function 
$f$ such that $f(j)$ is the $x$-coordinate of the north step with label $j$.

\begin{example}\label{ex:fn-to-path}
Let $\mcG=\Z$, $n=12$, $c=6$, and $m=2$. Consider the function $f$ with
values shown in Table~\ref{tab:ex-fn}.
Figure~\ref{fig:fn-to-path} shows the labeled path for $f$.
The area vector for the unlabeled path is 
$$g=(5,7,9,11,10,10,6,8,6,8,10,11).$$
The label word of $f$ is $w=(4,7,8,11,1,3,2,5,6,9,12,10)$.
\begin{table}[h]
  \centering
  \caption{Function analyzed in Example~\ref{ex:fn-to-path} and illustrated in 
Figure~\ref{fig:fn-to-path}.}
  \begin{tabular}{ccccccccccccc}\toprule
    $a$  & 1 & 2 & 3 & 4 & 5 & 6 & 7 & 8 & 9 & 10 & 11 & 12 \\\midrule
    $f(a)$ & 4 &12 & 6 & 1 &12 &16 & 1 & 1 &16 & 17 &  1 & 16 \\\bottomrule
  \end{tabular}
\label{tab:ex-fn}
\end{table}
\end{example}
\begin{figure}[h]
\begin{center}
  \includegraphics[width=0.9\linewidth]{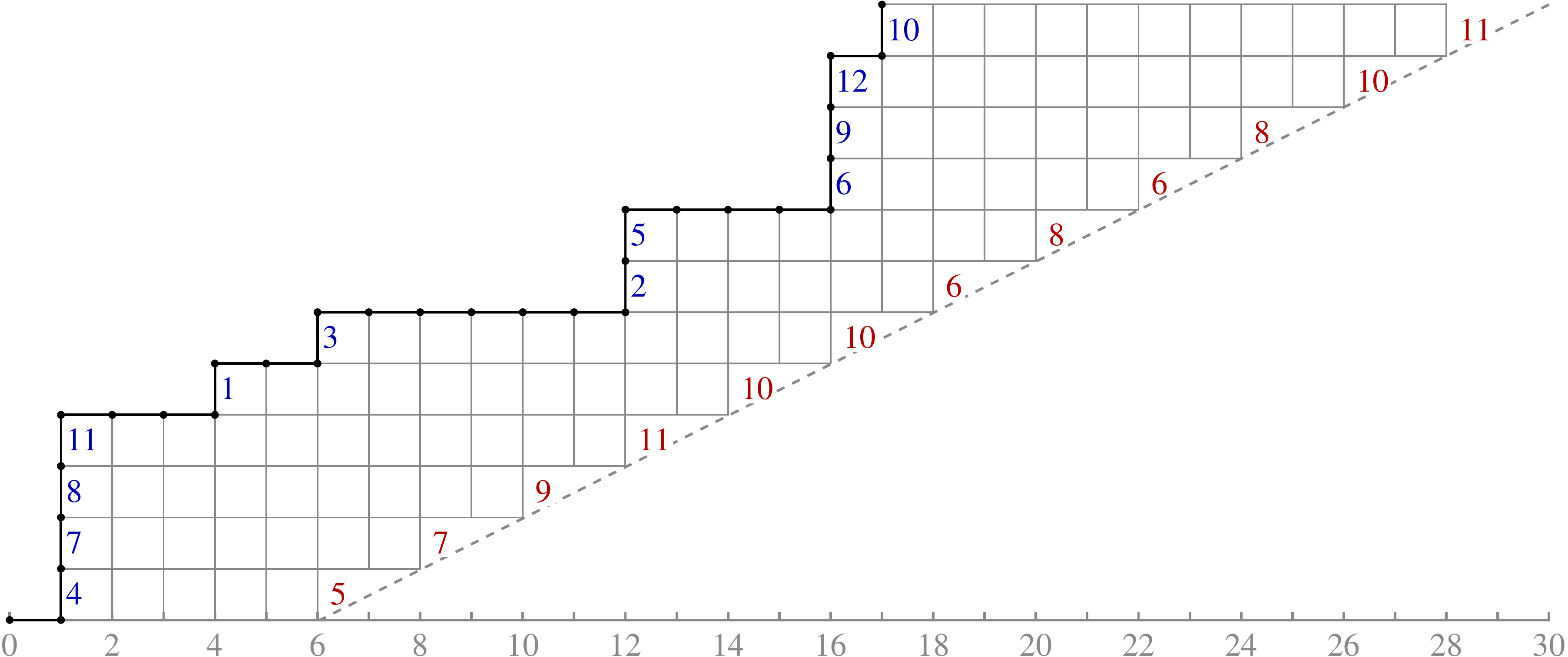}
  \caption{Labeled path for the function in
    Example~\ref{ex:fn-to-path}. Labels are placed in blue to the left
    of each north step. Entries of the $g$-vector are listed in each
    row in red along the line $x=2y+6$.}
\label{fig:fn-to-path}
\end{center}
\end{figure}

For any path $\pi=\{(x_i,i)\}\in P_n$, a \emph{run of north steps} 
in $\pi$ is a maximal interval $\{i,i+1,\ldots,j\}$ of
consecutive indices such that $x_i=x_{i+1}=\cdots=x_j$. The \emph{length}
of this run is $j-i+1$, which is the number of north steps of $\pi$ lying
on the line $x=x_i$. Let $\run(\pi)$
be the multiset consisting of the lengths of all runs of north steps in $\pi$.
Suppose $\pi$ has area vector $G(\pi)=(g_0,g_1,\ldots,g_{n-1})$.
By definition of $G$, $x_i=x_{i+1}$ iff $g_{i+1}=g_i+m$.
A \emph{run} in an area vector $g$ is a maximal
subsequence of consecutive entries in which each successive entry exceeds
the previous one by $m$. Let $\run(g)$ be the multiset of run lengths
in $g$. For all $\pi\in P_n$, we have $\run(\pi)=\run(G(\pi))$.

An unlabeled path $\pi\in P_n$ may have the form $\pi(f)$ for
several different functions $f\in F_n$. The run structure of $\pi$
determines how many such $f$ there are, as follows.

\begin{prop}
Suppose $\pi\in P_n$ has $\run(\pi)=[r_1,r_2,\ldots,r_a]$. 
The number of $f\in F_n$ with $\pi(f)=\pi$ is the multinomial coefficient 
$\binom{n}{r_1,r_2,\ldots,r_a}$. 
\end{prop}
\begin{proof}
Choose distinct $z_1,z_2,\ldots,z_a$ such that $\pi$ has $r_i$ north steps 
with $x$-coordinate $z_i$. A function $f\in F_n$ has $\pi(f)=\pi$
iff the word $f(1),f(2),\ldots,f(n)$ is a rearrangement 
of $r_1$ copies of $z_1$, $r_2$ copies of $z_2$, and so on.
The number of such rearrangements is $\binom{n}{r_1,r_2,\ldots,r_a}$.
\end{proof}

\begin{example}
The path $\pi$ shown in Figure~\ref{fig:fn-to-path} 
has run multiset $\run(\pi)=[4,1,1,2,3,1]$. This path is $\pi(f)$
for $\binom{12}{4,1,1,2,3,1}=1663200$ choices of $f\in F_n$.
The area vector $g=G(\pi)$ has the same run multiset as $\pi$. 
We write $g=(\underline{5,7,9,11},\underline{10\zcomm},\underline{10\zcomm},
  \underline{6,8},\underline{6,8,10},\underline{11\zcomm}),$
where entries in the same run are underlined.
\end{example}

\subsection{Main Result for Skeletal Functions}
\label{subsec:skel-fn}

A function $f\in F_n$ is called \emph{$k$-skeletal}
(for parameters $c$ and $m$) iff the unlabeled path $\pi(f)$ is $k$-skeletal.
Similarly, a labeled path $(\pi,w)$ is $k$-skeletal iff $\pi$ is $k$-skeletal.
Let $\skf_k$ be the set of $k$-skeletal functions in $F_n$ with
parameters $c$ and $m$ and values in $\mcG$.

\begin{theorem}\label{thm:main-skf}
(a) For all $k,k'\in\{0,1,\ldots,n-1\}$, there is a canonical bijection
from $\skf_k$ to $\skf_{k'}$. 
(b) When $\mcG=\Z$, for all $k\in\{0,1,\ldots,n-1\}$, we have
$|\skf_k|=|\skf_{n-1}|=c(mn+c)^{n-1}$. 
\end{theorem}

The equality $|\skf_{n-1}|=c(mn+c)^{n-1}$ in~(b) 
is a classical result (compare to the proof of Corollary~\ref{cor:main-enum}). 
Specifically, $\skf_{n-1}$ is the set of
\emph{trapezoidal parking functions}, which are
functions $f$ whose path $\pi(f)$ stays in the trapezoidal region
$\{(x,y)\in\R^2: 0\leq y\leq n, 0\leq x\leq my+c-1\}$. The stated
formula for the number of such functions can be found, for example,
in~\cite[Theorem 1.2]{rstan-park}.

\subsection{Run Structure of Equivalence Classes}
\label{subsec:run-equiv}

The proof of Theorem~\ref{thm:main-skf} uses the set $S$ 
and equivalence relation $\sim$ from Section~\ref{subsec:cycle-area-vec}.

\begin{lemma}\label{lem:run}
Let $T$ be an equivalence class of $\sim$ in $S$ with unique Dyck
representative $g^+$. For every $h\in T$, $\run(h)=\run(g^+)$.
\end{lemma}
\begin{proof}
Let $T=\{g^+\succ g^1\succ g^2\succ\cdots\}$ be the finite sequence of all 
area vectors in $T$ ordered by decreasing area. 
By induction, it suffices to show that
if $g,h$ are two consecutive objects in this sequence, then $\run(g)=\run(h)$.

Let $g=(g_0,g_1,\ldots,g_{n-1})\in T$. To reach $h$ from $g$, we must apply
$C$ one or more times. Choose the largest $s$ so that the first $s$
entries of $g$ are $\leq c$;
we have $s\geq 1$ since $g_0\leq c$. When $C$ cycles these entries to the
end, they each become $\leq 0$. So $C^1(g),\ldots,C^s(g)$ are not in $S$.
If $s=n$, then applying $C$ additional times never leads to an object in $S$,
contradicting the fact that $h$ follows $g$ in the sequence. So $0<s<n$.
Next, choose the largest $t$ so that the $t$ entries of $g$ 
scanning forward from $g_s$ are $>c$. We must have $t>0$ because $s<n$.
Since the first entry of any vector in $S$ must be $\leq c$, we must cycle 
all $t$ of these entries to the end to reach the next object in $S$.
Thus, $h=C^{s+t}(g)$ where $0<s+t\leq n$. In the case $s+t=n$,
$h=(g_0-c,g_1-c,\ldots,g_{n-1}-c)$, which certainly has the same run 
multiset as $g$. In the case $s+t<n$, 
\begin{equation}\label{eq:cycled-runs}
 h=(g_{s+t},\ldots,g_{n-1};g_0-c,\ldots,g_{s+t-1}-c). 
\end{equation}
The key point is that $g_{s+t-1}>c$ but $g_{s+t}\leq c$,
so $g_{s+t}$ cannot be $g_{s+t-1}+m$. So $(g_0,\ldots,g_{s+t-1})$
is a union of certain runs of $g$, and the run lengths in this
part are unaffected when we subtract $c$ from every entry.
Similarly, $(g_{s+t},\ldots,g_{n-1})$ contributes the same run lengths
to $g$ and to $h$, once we notice (using the last sentence of the
proof of Lemma~\ref{lem:area-vec}) that $g_{n-1}$ and $g_0-c$ 
cannot belong to the same run in $h$. So $\run(h)=\run(g)$.
\end{proof}

\begin{example}
In Example~\ref{ex:find-pos}, all vectors $g^+,g',g,g''$ in the 
equivalence class have run multiset $[5,1,3,3,1,1]$. 
\end{example}

\begin{example}\label{ex:fn-to-path2}
Continuing Example~\ref{ex:fn-to-path}, 
the path in Figure~\ref{fig:fn-to-path} has Dyck area vector
\[ g=g^+=(\underline{5,7,9,11},\underline{10\zcomm},\underline{10\zcomm},
  \underline{6,8},\underline{6,8,10},\underline{11\zcomm})\in S. \] To
  reach the next element of $S$ equivalent to $g$, we must cycle the
  initial entries less than or equal to $c$ (namely the first entry $5$)
  and continue to cycle everything after that exceeding $c$ (namely
  $7$, $9$, $11$, $10$ and $10$). This leads to
\[ g'=(\underline{6,8},\underline{6,8,10},\underline{11\zcomm},
 \underline{-1,1,3,5},\underline{4\zcomm},\underline{4\zcomm}). \]
Repeating this process, we reach two more area vectors in the equivalence class:
\[ g''=(\underline{6,8,10},\underline{11\zcomm},
 \underline{-1,1,3,5},\underline{4\zcomm},\underline{4\zcomm},\underline{0,2}); \]
\[ g'''=(\underline{-1,1,3,5},\underline{4\zcomm},\underline{4\zcomm},\underline{0,2},
 \underline{0,2,4},\underline{5\zcomm}). \] We have $\pos(g)=12$,
 $\pos(g')=5$, $\pos(g'')=1$, $\pos(g''')=3$, and all these vectors
 have run multiset $[4,1,1,2,3,1]$.  So $g$ is $k$-skeletal for $5\leq
 k<12$, $g'$ is $k$-skeletal for $3\leq k<5$, $g''$ is not
 $k$-skeletal for any $k$, and $g'''$ is $k$-skeletal for $0\leq
 k<3$.  Note that exactly one of the elements $\{g,g',g'',g'''\}$ is
 $k$-skeletal for each $k$ in the range $0\leq k\leq 11$,
 as assured by Theorem~\ref{thm:main-skf}(a).
\end{example}

\subsection{Proof of Theorem~\ref{thm:main-skf}.}
\label{subsec:prove-thm-skf}

Fix $k,k'\in\{0,1,\ldots,n-1\}$ and $f\in\skf_k=\skf_k(\mcG;c,m)$. 
Let $g$ be the area vector of $\pi(f)$,
and let $g'$ be the unique $k'$-skeletal area vector with $g'\sim g$
(Theorem~\ref{thm:main-skp}).
We can go from $g$ to $g'$ by stepping up or down through the
equivalence class of $g$ using powers of $C^{-1}$ or $C$,
as described in the proof of Lemma~\ref{lem:run}.
Each step cycles the area vector in a way that preserves the run multiset.
So we can cycle the labels of the north steps in the same way without violating
the rules for labeled paths. For example, suppose we start with
$g=(g_0,g_1,\ldots,g_{n-1})$ and label sequence $w=(w_0,w_1,\ldots,w_{n-1})$. 
If the next step replaces $g$ by the $h$ shown in~\eqref{eq:cycled-runs},
then we replace the label sequence $w$ by 
$(w_{s+t},\ldots,w_{n-1};w_0,\ldots,w_{s+t-1})$. 
We eventually reach $g'$ and a new label sequence $w'$, 
which corresponds to some function $f'\in\skf_{k'}$. 
The map sending $f$ to $f'$ is the required bijection 
from $\skf_k$ to $\skf_{k'}$.

\begin{example}\label{ex:fn2}
Let $f$ be the function shown in Table~\ref{tab:ex-fn}.
The area vector of $\pi(f)$ is the vector $g$ from 
Example~\ref{ex:fn-to-path2}, which is $7$-skeletal,
so $f\in\skf_7$. Let us find the image of $f$ under
the bijection from $\skf_7$ to $\skf_4$.  The label sequence of $f$ is 
$w=(\underline{4,7,8,11},\underline{1\zcomm},\underline{3\zcomm},
     \underline{2,5},\underline{6,9,12},\underline{10\zcomm}).$
When $g$ cycles to $g'$, the label sequence cycles to
$w'=(\underline{2,5},\underline{6,9,12},\underline{10\zcomm},
\underline{4,7,8,11},\underline{1\zcomm},\underline{3\zcomm}).$
The labeled path encoded by $g'$ and $w'$ is shown in 
Figure~\ref{fig:fn-to-path2}. The corresponding $4$-skeletal function
$f'$ has the values shown in Table~\ref{tab:ex-fn2}.
\end{example}
\begin{table}[h]
  \centering
  \caption{Function analyzed in Example~\ref{ex:fn2} and illustrated in Figure~\ref{fig:fn-to-path2}.}
  \begin{tabular}{ccccccccccccc}\toprule
    $a$    & 1 & 2 & 3 & 4 & 5 & 6 & 7 & 8 & 9 & 10 & 11 & 12 \\\midrule
    $f'(a)$&22 & 0 &24 & 19& 0 & 4 & 19& 19& 4 &  5 &  19&  4 \\\bottomrule
  \end{tabular}
  \label{tab:ex-fn2}
\end{table}
\begin{figure}[h]
\begin{center}
  \includegraphics[width=0.9\linewidth]{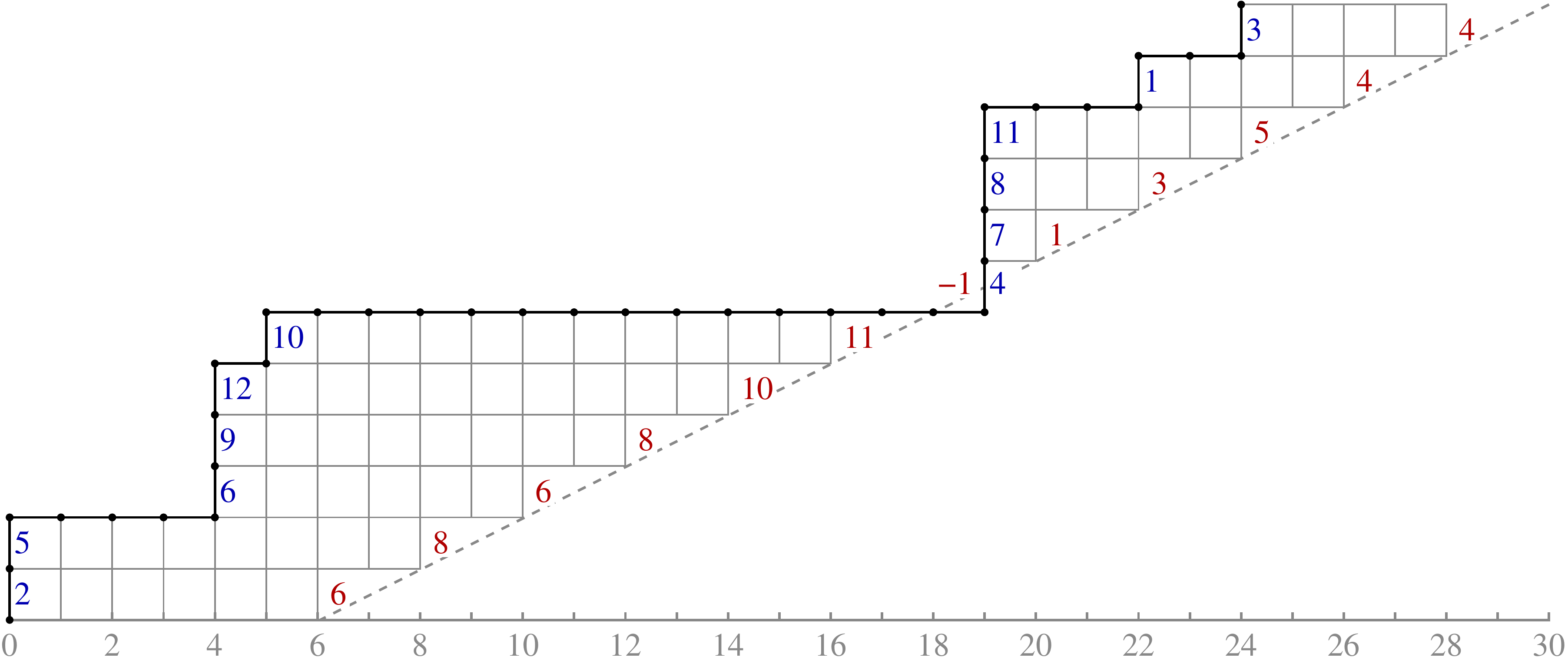}
\caption{Labeled path for the function $f'$ in Example~\ref{ex:fn-to-path2}.}
\label{fig:fn-to-path2}
\end{center}
\end{figure}

\begin{remark}
The symmetric group $\mathfrak{S}_n$ acts on each set $\skf_k$ by
permuting the inputs: $\sigma\cdot f=f\circ\sigma^{-1}$ for
$\sigma\in\mathfrak{S}_n$ and $f\in\skf_k$. To see that $\sigma\cdot f$
does belong to $\skf_k$, note that $f$ and $\sigma\cdot f$ have the
same multiset of output values. So the (unlabeled) path of $\sigma\cdot f$ 
equals the path of $f$, which is $k$-skeletal by assumption on $f$.
\end{remark}

\section{$\mcG$-Valued Chip Firing and Skeletal Functions}
\label{sec:skel-chip}

This section introduces a generalization of chip firing relative to an
additive subgroup $\mcG$ of $\R$ (see~\cite{levine-peres}). We first
explain how this framework allows for a chip-firing interpretation of
the skeletal functions from Section~\ref{sec:skel-fns}. Then we
discuss how the general theory from Section~\ref{subsec:motivation}
can be adapted to the setting of chip configurations with values in
$\mcG$.

We continue to assume this setup: $\mcG$ is a fixed additive subgroup of $\R$,
$n$ is a positive integer, and $c,m\in\mcG$ are parameters
with $c>0$ and $m\geq 0$. We study the following chip-firing
model built from these parameters.  Let $K_{n+1}$ be the complete graph 
with vertex set $\{0,1,2,\ldots,n\}$, where $0$ is a special vertex called
the \emph{sink}.
 For each $i\neq j$ between $1$ and $n$, the edge from $i$ to $j$
has \emph{capacity} $m$. For each $i$ between $1$ and $n$,
the edge from $i$ to $0$ has \emph{capacity} $c$.
A \emph{chip configuration} on $K_{n+1}$ with values in $\mcG$ 
is a function $\ch:[n]\rightarrow\mcG$, where $[n]=\{1,2,\ldots,n\}$.
We think of $\ch(i)$ as
the chip count at vertex $i$, which might be negative or non-integral
(depending on $\mcG$). Our chip configurations do not record the chip
count at the sink vertex $0$. The chip configuration $\ch$
is \emph{nonnegative} (written $\ch\geq 0$) 
iff $\ch(i)\geq 0$ for all $i\in [n]$.

We introduce an operation $\phi_i$ on chip configurations called
\emph{firing at vertex $i$}. By definition, $\phi_i(\ch)$ is the configuration
obtained from $\ch$ by decreasing $\ch(i)$ by $m(n-1)+c$
and increasing $\ch(j)$ by $m$ for all $j\neq i$ in $[n]$.
Intuitively, when vertex $i$ fires, it sends $m$ chips along each of
the $n-1$ edges to other $j\in [n]$, and it sends $c$ chips to the sink.
We say a vertex $i\in [n]$ can \emph{legally fire} iff $\ch(i)\geq m(n-1)+c$.
If $\ch\geq 0$ and vertex $i$ can legally fire, then $\phi_i(\ch)\geq 0$.

More generally, suppose $S$ is a nonempty subset of $[n]$ of size $f$.
The operation $\phi_S$ (\emph{firing at vertex set $S$}) acts on any $\ch$ 
by decreasing $\ch(i)$ by $m(n-f)+c$ for each $i\in S$ and
increasing $\ch(j)$ by $mf$ for all $j\in [n]\setminus S$.
We say $S$ can \emph{legally fire} in configuration $\ch$ 
iff $\ch(i)\geq m(n-f)+c$ for all $i\in S$. 
If $\ch\geq 0$ and subset $S$ can legally fire, then $\phi_S(\ch)\geq 0$.

Next, let $T$ be a nonempty subset of $[n]$ of size $p$.
The operation $\beta_T$ (\emph{borrowing at vertex set $T$}) acts on any $\ch$ 
by increasing $\ch(i)$ by $m(n-p)+c$ for each $i\in T$ and
decreasing $\ch(j)$ by $mp$ for all $j\in [n]\setminus T$.
Evidently, $\beta_T$ is the two-sided inverse of $\phi_T$
when acting on all (not necessarily nonnegative) chip configurations.
We say $T$ can \emph{legally borrow} in configuration $\ch$ 
iff $\ch(j)\geq mp$ for all $j\in [n]\setminus T$.
If $\ch\geq 0$ and subset $T$ can legally borrow, then $\beta_T(\ch)\geq 0$.

A chip configuration $\ch$ is a function from $[n]$ to $\mcG$,
so we can represent $\ch$ as a labeled path as explained 
in~\S\ref{subsec:rep-fn-path}. We use the following terminology in
this setting. Given $v,z\in [n]$, say that $v$ is \emph{poorer} than $z$ 
and $z$ is \emph{richer} than $v$ if $\ch(v)<\ch(z)$. To make
the labeled path for $\ch$, place vertices in $n$ rows 
with poorer vertices occupying lower rows.
Each vertex $v\in [n]$ with chip count $\ch(v)$ labels
a north step on the line $x=\ch(v)$. If several vertices 
have equal chip count, arrange them in the same column with vertex
labels increasing from bottom to top. Finally, connect the $n$
north steps with east steps to get a path proceeding northeast.
If $\ch\geq 0$, then this path may start at the origin.
For any chip configuration $\ch$, let $\pi(\ch)=\{(x_i(\ch),i)\}$
be the associated path with label sequence $w(\ch)=(w_0,\ldots,w_{n-1})$;
so $\ch(w_i)=x_i(\ch)$ for all $i$.
Let $G(\ch)$ be the area vector of $\pi(\ch)$
(computed relative to the reference line $x=my+c$).

\begin{example}\label{ex:chip}
Let $n=6$, $m=2$, $c=4$, and $\ch=(\ch(1),\ldots,\ch(6))=(4,1,5,5,14,8)$. Then
$\pi(\ch)=\{(1,0),(4,1),(5,2),(5,3),(8,4),(14,5)\}$, $w(\ch)=(2,1,3,4,6,5)$,
and $G(\ch)=(3,2,3,5,4,0)$. We may identify $\pi(\ch)$ with the lattice
path $\E\N\E\E\E\N\E\N\N\E\E\E\N\E\E\E\E\E\E\N$, where
the north steps are labeled $2,1,3,4,6,5$ from bottom to top.
In $\ch$, vertex 5 can fire since it has 14 chips. 
Firing vertex $5$ would change $\ch$ to $\phi_5(\ch)=(6,3,7,7,0,10)$.
In $\ch$, no $2$-element subset can legally fire since both vertices
in the subset would need at least $12$ chips. Similarly,
for all $f\geq 2$, no $f$-element subset can legally fire in $\ch$. 
\end{example}

\subsection{Skeletal Chip Configurations}
\label{subsec:skel-chip}

For $k\in\{0,1,\ldots,n-1\}$, a chip configuration $\ch$ is called
\emph{$k$-skeletal} iff these conditions hold:
\begin{itemize}
\item[(C0)] $\ch\geq 0$.
\item[(C1)] For all $S\subseteq [n]$, if $0<|S|\leq k+1$, then
 $S$ cannot legally fire in configuration $\ch$.
\item[(C2)] For all nonempty $T\subseteq [n]$, if $T$ can legally
borrow in configuration $\ch$, then there exists $S\subseteq [n]$
such that $0<|S|\leq k+1$ and $S$ can legally fire in configuration
$\beta_T(\ch)$.
\end{itemize}
Condition~(C2) says that~(C1) fails for every configuration reachable
from $\ch$ by a legal borrow move.  \footnote{Inspired by
conversations with the first author and Sam Hopkins,
Dochtermann~\cite{Dochtermann} investigated functions which satisfy
(C0) and (C1), but not the magic (C2).  When one drops condition (C2)
the resulting objects are determined by monomial ideals and thus
amenable to techniques from commutative algebra.  This line of inquiry
was further pursued by Dochtermann \& King~\cite{dochtermann2021trees}
as well as by Kumar, Lathar, and Roy~\cite{kumar2021skeleton,
  kumar2022standard, roy2020standard}.}  Let $\skc_k$ be the set of
$k$-skeletal chip configurations.

\begin{theorem}\label{thm:main-skc}
For all $k$, $\skc_k=\skf_k$. Thus, $|\skc_k|$ is independent of $k$.
In the case $\mcG=\Z$, $|\skc_k|=c(mn+c)^{n-1}$.
\end{theorem}

We prove this theorem by showing that Conditions~(C0),~(C1), and~(C2)
on a chip configuration $\ch$ are equivalent to Conditions~(A0), (A1$'$), 
and~(A2$'$) on the area vector $g=G(\ch)$  (see Proposition~\ref{prop:A012'}).
Condition~(C0) for $\ch$, Condition~(P0) for $\pi(\ch)$,
and Condition~(A0) for $g$ are equivalent. To compare Conditions~(C1) 
and~(A1$'$), fix $f$ with $0<f\leq n$.  The following statements are 
equivalent for a configuration $\ch\geq 0$ with area vector $g$:
\begin{itemize}
\item[(a)] 
 There exists an $f$-element subset $S$ of $[n]$ that can legally fire in $\ch$.
\item[(b)] The set of $f$ richest vertices in $[n]$ can legally fire in $\ch$.
\item[(c)] The $f$\nobreakdash-th richest vertex in $\ch$ has at least $m(n-f)+c$ chips.
\item[(d)] $x_{n-f}(\ch)\geq m(n-f)+c$.
\item[(e)] $m(n-f)+c-g_{n-f}\geq m(n-f)+c$. 
\item[(f)] $g_{n-f}\leq 0$.
\end{itemize}
Negating this, no $f$-element subset $S$ of $[n]$ can legally fire 
in configuration $\ch$ iff $g_{n-f}>0$.  Applying this to $f=1,2,\ldots,k+1$, 
we see that Condition~(C1) for $\ch$ is equivalent to Condition~(A1)
for $g$, which is equivalent to~(A1$'$).

\subsection{Analysis of Condition~(C2)}
\label{subsec:analyze-C2}

Continuing the proof of Theorem~\ref{thm:main-skc}, the next step
is to develop a simplified version of Condition~(C2). Let $\ch$ be a chip 
configuration satisfying~(C0) and~(C1), and let $g$ be the associated area 
vector.  For $p\in\{1,2,\ldots,n\}$, we first determine when there exists a
$p$-element subset $T$ of $[n]$ such that $T$ can legally borrow in $\ch$.
In the case $p<n$, the following statements are equivalent:

\begin{itemize}
\item[(a)] There is $T\subseteq [n]$ with $|T|=p$ such that 
$T$ can legally borrow in $\ch$.
\item[(b)] There exist $n-p$ vertices in $[n]$ that each have
 at least $mp$ chips in $\ch$.
\item[(c)] The richest $n-p$ vertices in $\ch$ each have
 at least $mp$ chips.
\item[(d)] The $(p+1)$\nobreakdash-th poorest vertex in $\ch$ has at least $mp$ chips.
\item[(e)] $x_p(\ch)\geq mp$.
\item[(f)] $mp+c-g_p\geq mp$.
\item[(g)] $g_p\leq c$.
\end{itemize}
When $p=n$, the set $T=[n]$ can always legally borrow. So, for all $p$:
\begin{equation}\label{eq:borrow-cond}
\text{(some $p$-element subset $T$ of $[n]$ can legally borrow in $\ch$)
iff ($p=n$ or $g_p\leq c$).}
\end{equation}

\begin{example}
In Example~\ref{ex:chip}, $m=2$ and $\ch$ has area vector $g=(3,2,3,5,4,0)$.
We can find subsets of size $1,2,4,5,6$ that can legally borrow.
But no such subset exists of size $p=3$, since $g_3=5>4=c$.
Indeed, vertices $1,2,3,4$ each have fewer than $6$ chips,
so $T$ would need to include all of these vertices for $\beta_T$
to be legal. But then $T$ has size larger than $3$.
\end{example}

Still holding $p$ fixed,
we can reduce the number of $p$-element subsets $T$ 
that we must consider when checking Condition~(C2). 
This condition involves the following IF-statement: 
``if $\beta_T$ is legal for configuration $\ch$,
 then configuration $\beta_T(\ch)$ has a nonempty subset $S$ of 
 size at most $k+1$ that can legally fire.''
This IF-statement is automatically true in certain cases, listed next.
\begin{itemize}
\item[(i)] In the case $p<n$ and $g_p>c$: 
 $\beta_T$ is not legal (by~\eqref{eq:borrow-cond}), 
so the IF-statement is true.
\item[(ii)] In the case where $\beta_T$ is legal 
(i.e., $p=n$ or $g_p\leq c$) and 
some $v\in T$ has $\ch(v)\geq m(p-1)$: borrowing at $T$ 
increases the chip count at $v$ to be $\geq m(n-1)+c$.
So the IF-statement can be fulfilled by taking $S=\{v\}$.
\item[(iii)] In the case ($p=n$ or $g_p\leq c$) and $g_{p-1}\leq c$:
 the $p$\nobreakdash-th poorest vertex in $[n]$ has chip count
 $x_{p-1}(\ch)=m(p-1)+c-g_{p-1}\geq m(p-1)$.
 As $T$ has size $p$, some vertex in $T$ must have at least $m(p-1)$ chips.
 Thus, the preceding case~(ii) applies, and the IF-statement
 in~(C2) is true.
\item[(iv)] In the case ($p=n$ or $g_p\leq c$) and $g_{p-1}>c$
 and $T$ is not the set $\{w_0,\ldots,w_{p-1}\}$ of $p$ poorest vertices
 in $\ch$: $T$ must contain some $v$ outside this set, so $p<n$, $g_p\leq c$,
 and $\ch(v)\geq \ch(w_p)=x_p(\ch)=mp+c-g_p\geq mp\geq m(p-1)$. 
 So case~(ii) applies, and the IF-statement in~(C2) is true.
\end{itemize}

Removing all these cases, and letting $p$ vary,
we conclude that Condition~(C2) for
configuration $\ch$ is equivalent to the following condition:
\begin{quote}
(C2$'$)~For all $p\in\{1,2,\ldots,n\}$, if ($p=n$ or $g_p\leq c$)
and $g_{p-1}>c$ and $T=\{w_0,\ldots,w_{p-1}\}$, 
then configuration $\beta_T(\ch)$ fails Condition~(C1).
\end{quote}

\subsection{Equivalence of~(C2$'$) and~(A2$'$)}
\label{subsec:equiv-C2-A2}

\begin{lemma}\label{lem:Cp-area-vec}
Suppose $\ch$ satisfies Conditions~(C0) and~(C1), and $p$ and $T$
satisfy the hypothesis of ($C2'$). If $\ch$ has area vector 
$g=(g_0,g_1,\ldots,g_{n-1})$, then $\beta_T(\ch)$ has area vector 
$C^p(g)=(g_p,\ldots,g_{n-1},g_0-c,\ldots,g_{p-1}-c)$.
\end{lemma}
\begin{proof}
We are assuming $T=\{w_0,\ldots,w_{p-1}\}$ is the set of $p$ poorest vertices 
in $\ch$. Let $T'=[n]\setminus T=\{w_p,\ldots,w_{n-1}\}$ be the set of $n-p$
richest vertices in $\ch$.  For all $v\in T$, $\ch(v)\geq 0$.
Borrowing at $T$ shifts all chip counts in $T$ up by $m(n-p)+c$,
so $\beta_T(\ch)(v)\geq m(n-p)+c$ for all $v\in T$.
On the other hand, for all $z\in T'$, $\ch(z)\leq \ch(w_{n-1})
 =x_{n-1}(\ch)=m(n-1)+c-g_{n-1}<m(n-1)+c$.
Borrowing at $T$ shifts all chip counts in $T'$ down by $mp$,
so $\beta_T(\ch)(z)<m(n-p-1)+c$ for all $z\in T'$.
It follows that the list of vertices in $\beta_T(\ch)$,
from poorest to richest, is $(w_p,\ldots,w_{n-1},w_0,\ldots,w_{p-1})$.
This is the label sequence $w(\beta_T(\ch))$.

Let $\beta_T(\ch)$ have area vector $g^*=(g_0^*,g_1^*,\ldots,g_{n-1}^*)$ 
and path $\pi(\beta_T(\ch))=\{(x_i^*,i):0\leq i<n\}$.
We must prove $(g_0^*,\ldots,g_{n-1}^*)
=(g_p,\ldots,g_{n-1},g_0-c, \ldots,g_{p-1}-c)$.
For $i$ in the range $0\leq i<n-p$, 
$x_i^*=\ch(w_{p+i})-mp=x_{p+i}(\ch)-mp,$ so
\[ g_i^*=mi+c-x_i^*=m(p+i)+c-x_{p+i}(\ch)=g_{p+i}. \]
For $i$ in the range $n-p\leq i<n$,
$x_i^*=\ch(w_{i-(n-p)})+m(n-p)+c=x_{i-(n-p)}(\ch)+m(n-p)+c$, so
\[ g_i^*=mi+c-x_i^*=m(i-(n-p))-x_{i-(n-p)}(\ch)=g_{i-(n-p)}-c. \qedhere \]
\end{proof}

\begin{example}
In Example~\ref{ex:chip}, we can take $p=4$, $T=\{2,1,3,4\}$,
and apply $\beta_T$ to $\ch$ to get $\beta_T(\ch)=(12,9,13,13,6,0)$. 
This converts $g=G(\ch)=(3,2,3,5,4,0)$ to 
$g^*=G(\beta_T(\ch))=(4,0,-1,-2,-1,1)$.
\end{example}

We finish the proof of Theorem~\ref{thm:main-skc} by
comparing the conclusions of Conditions~(C2$'$) and~(A2$'$).
The statement ``$\beta_T(\ch)$ fails Condition~(C1)'' 
is equivalent to ``$G(\beta_T(\ch))$ fails Condition~(A1$'$),''
which (by Lemma~\ref{lem:Cp-area-vec})
 is equivalent to ``$\pos(C^p(G(\ch)))\leq k$.''
Thus, Condition~(C2$'$) for $\ch$ is equivalent to Condition~(A2$'$)
for $G(\ch)$.

\subsection{General $\mcG$-Valued Chip Firing}
\label{subsec:gen-chip}

So far, we have studied chip configurations with values in $\mcG$
only for certain weighted complete graphs with a sink vertex.
Here we briefly indicate how to extend the general setup of
Section~\ref{subsec:motivation} to the setting of $\mcG$-valued chip 
configurations.

Let $\mcG$ be an additive subgroup of $\mathbb{R}$.  
A \emph{$\mcG$-weighted graph} is a simple graph $G=(V,E)$
together with a \emph{weight function} $\wt:E\rightarrow \mcG\cap\R_{>0}$ 
that assigns a positive weight in $\mcG$ to each edge of $G$.
Let $(v_1,v_2,\ldots,v_N)$ be a fixed total ordering of $V$.
The \emph{Laplacian} of the $\mcG$-weighted graph $G$ is the 
$N\times N$ matrix $\LL$ with entries
\begin{equation*}
\LL_{i,j}=\begin{cases}
          \,\,\sum_{e:\ \text{$e$ touches $v_i$}} \wt(e)
         \quad & \text{if}\,\, i=j;  \\
          \,\, -\wt(e)  \quad & \text{if}\,\,  i\neq j  \,\, 
           \text{and}\,\,  e = \{v_i,v_j\}\in E.
     \end{cases}
\end{equation*} 

The restriction to simple graphs is no real loss of generality.
To model graphs with multiple edges between the same two vertices,
we combine all those edges into a single edge whose weight is the
sum of the original edge weights. This modification produces a simple
graph with the same Laplacian as the original graph.

A \emph{$\mcG$-valued chip configuration} is a function 
$\ch:V \rightarrow \mcG$.  Chip configurations $\ch$ and $\ch^*$ are
\emph{chip-firing equivalent} if $\ch-\ch^*\in\LL(\Z^N)$.  
The \emph{chip-firing group} of the edge-weighted graph $G$ is
$\Pic(G,\wt) = \mcG^N/\LL(\Z^N)$. This group, which may no longer be
discrete, is the $\Z$-cokernel of the Laplacian. Define $\Pic^0(G,\wt)$ to
be the subgroup of $\Pic(G,\wt)$ consisting of classes represented
by degree-zero chip configurations. Given a sink $q\in V$ and
an abstract simplicial complex $\Delta$ on $V'=V\setminus\{q\}$,
we define \emph{$\Delta$-critical configurations} 
exactly as in Section~\ref{subsec:motivation}. 
It would be interesting to see if Theorem~\ref{thm:delta-critical}
extends to the setting of $\mcG$-valued chip configurations.
Our analysis of $k$-skeletal objects proves this result in the special
cases where $G$ is a complete graph (with sink and weights as described
earlier) and $\Delta$ is the $(k+1)$-skeleton of $V'$, namely the set of all
nonempty subsets of $V'$ of size $k+1$ or less. This is the origin of the
term ``$k$-skeletal'' for the various combinatorial collections considered
earlier.

\section{$t$-Analogues of Skeletal Paths and Skeletal Functions}
\label{sec:t-analog}

Throughout this section, we take $\mcG=\Z$ and fix positive integer
parameters $n,m,c$. A \emph{statistic} on area vectors is 
any function $\stat:\AV_n\rightarrow\Z_{\geq 0}$. 
A \emph{$t$-analogue of $|\skv_k|$} is the polynomial generating function
$\skv_k(t;\stat)=\sum_{g\in\skv_k} t^{\stat(g)}$. 

We prove the following $t$-analogue of Theorem~\ref{thm:main-skp}
and Corollary~\ref{cor:main-enum}.

\begin{theorem}\label{thm:skv-t}
Suppose $F:\R\rightarrow\R$ is a function satisfying $F(z)=F(c-z)$
for all $z\in\R$, and $\stat:\AV_n\rightarrow\Z_{\geq 0}$ satisfies
\begin{equation}\label{eq:stat-hyp}
\stat(g_0,g_1,\ldots,g_{n-1})=\sum_{0\leq i<j<n} F(g_i-g_j).
\end{equation}
Then for all $k\in\{0,1,\ldots,n-1\}$, $\skv_k(t;\stat)=\skv_{n-1}(t;\stat)$.
This is a $t$-analogue of the $m$-ballot number
$\frac{c}{(m+1)n+c}\binom{(m+1)n+c}{n}$ that is independent of $k$.
\end{theorem}
\begin{proof}
Let $g=(g_0,g_1,\ldots,g_{n-1})$ be any area vector. 
Since $C(g)=(g_1,\ldots,g_{n-1},g_0-c)$, we have
\begin{align*}
 \stat(C(g)) &= \sum_{1\leq i<j<n} F(g_i-g_j) + \sum_{0<j<n} F(g_j-(g_0-c))\\
 &= \sum_{1\leq i<j<n} F(g_i-g_j) + \sum_{0<j<n} F(c-(g_0-g_j))\\
 &= \sum_{1\leq i<j<n} F(g_i-g_j) + \sum_{0<j<n} F(g_0-g_j)
 = \sum_{0\leq i<j<n} F(g_i-g_j)=\stat(g).
\end{align*}
Therefore, any area vector $h$ reachable from $g$ by applying powers of $C$
has $\stat(h)=\stat(g)$. Recall (from Theorem~\ref{thm:skp-eqv-class} and the 
following statement) that there is a bijection $\skv_k\rightarrow\skv_{n-1}$
sending $g\in\skv_k$ to the unique Dyck area vector $h\in\skv_{n-1}$
with $g\sim h$. Using this bijection, we compute
\[ \skv_k(t;\stat)=\sum_{g\in\skv_k} t^{\stat(g)}
                  =\sum_{h\in\skv_{n-1}} t^{\stat(h)}=\skv_{n-1}(t;\stat).
\qedhere \]
\end{proof}

Some statistics that arise in the theory of $q,t$-Catalan numbers have
the form~\eqref{eq:stat-hyp}. For example, taking $c=m=1$,
Haglund's celebrated combinatorial formula~\cite{hag-mac-conj,hag-book} 
for the $q,t$-Catalan numbers 
can be written $\Cat_n(q,t)=\sum_{g\in\skv_{n-1}} q^{\area(g)}t^{\dinv(g)}$,
where $\dinv(g_0,g_1,\ldots,g_{n-1})$ is the number of pairs $i<j$
with $g_i-g_j\in\{0,1\}$. This has the form~\eqref{eq:stat-hyp} if
we take $F$ to be the characteristic function of $\{0,1\}$, namely
$F(0)=F(1)=1$ and $F(z)=0$ for all other real $z$. Since $\dinv$ and
$\area$ are known to be equidistributed on Dyck vectors, 
Theorem~\ref{thm:skv-t} shows that for every $k$,
\[ \sum_{g\in\skv_k} q^{\dinv(g)}=\sum_{h\in\skv_{n-1}} q^{\area(h)}, \]
which is the $q$-analogue of the Catalan numbers first studied
by F\"urlinger and Hofbauer~\cite{fur-hof}.

For parameters $m=1$ and $c\geq 1$, we can define $\dinv^{(c)}(g)$ to
be the number of pairs $i<j$ with $0\leq g_i-g_j\leq c$. In this variation,
$F(z)=1$ for $0\leq z\leq c$ and $F(z)=0$ otherwise.
For parameters $c=1$ and $m\geq 1$, we can use the function
$F(z)=\max(0,m+1/2-|z-1/2|)$. On integer inputs, $F(z)=m+1-z$
for $1\leq z\leq m$, $F(z)=m+z$ for $-m\leq z\leq 0$, and $F(z)=0$
otherwise. So this $F$ recovers the ``slope-$m$ $\dinv$ statistic'' 
used to define higher-order $q,t$-Catalan numbers~\cite[Sec.~2.1]{loehr-mcat}.
For general $c,m\in\Z_{>0}$, we can take $F$ to be the piecewise linear
function that goes from $(-m,0)$ to $(0,m)$ to $(c,m)$ to $(c+m,0)$
and is $0$ outside this range. Or we may use $F(z)=\max(0,m+c/2-|z-c/2|)$.

The $q,t$-analogue $\skv_k(q,t)=\sum_{g\in\skv_k} q^{\area(g)}t^{\dinv(g)}$ 
is not independent of $k$. It would be interesting to compare the combinatorial
properties of these polynomials to $q,t$-Catalan polynomials and their
generalizations. Alternatively, perhaps there is a $k$-skeletal variation
of the area statistic on $\skv_k$ which, when paired with $\dinv$, does give
the $q,t$-Catalan polynomial for every $k$.

A result similar to Theorem~\ref{thm:skv-t} holds for $t$-analogues of
$k$-skeletal functions. 

\begin{theorem}\label{thm:skf-t}
Extend $C$ to act on pairs consisting of an
area vector $g$ and a label sequence $w$ by writing
\[ C((g_0,g_1,\ldots,g_{n-1}),(w_0,w_1,\ldots,w_{n-1}))
    =((g_1,\ldots,g_{n-1},g_0-c),(w_1,\ldots,w_{n-1},w_0)).\]
If $\stat$ is a statistic on labeled paths that satisfies
$\stat(C(g,w))=\stat(g,w)$ for all inputs $(g,w)$, then
$\sum t^{\stat(g,w)}$ (summed over all pairs $(g,w)$ encoding
functions in $\skf_k$) is a $t$-analogue of $|\skf_k|$ that is 
independent of $k$. 
\end{theorem}

This theorem follows from the bijective proof of
Theorem~\ref{thm:main-skf} (Section~\ref{subsec:prove-thm-skf}).

For example, let $m=c=1$ and define $\dinv(g,w)$ to be the number
of pairs $i<j$ with $g_i-g_j=0$ and $w_i<w_j$, or $g_i-g_j=1$ and $w_i>w_j$
(see~\cite[Chapter 5]{hag-book}).
It is easy to check that $\dinv$ is preserved by $C$.
The $q,t$-analogue of classical parking functions
studied in $q,t$-Catalan theory is $\sum_{(g,w)} q^{\area(g)}t^{\dinv(g,w)}$
where we sum over labeled Dyck paths of height $n$.
Setting $q=1$ and summing over labeled paths for $k$-skeletal functions
gives a $t$-analogue that is independent of $k$. Here too, perhaps there
are variations of area (depending on $k$) that pair with $\dinv$ to 
give the same $q,t$-polynomial for every $k$.

\section{First-Return Recursion and Bijection}
\label{sec:first-return}

Corollary~\ref{cor:main-enum} provides an enumeration of the number of
$k$-skeletal paths when $\mcG=\Z$. Here we give an alternative proof
when $c=1$ using a generalization of the first-return recursion for
the Catalan numbers.

For any $n\geq 0$ we identify elements of $P_n(\mathbb{Z})$ with the
classical lattice paths starting at the origin. We write $\Dmngen$ for
the set of \emph{augmented $m$-Dyck paths of height $n$}, that is, the
subset of paths in $P_n(\mathbb{Z})$ that remain weakly above the line
$x=my$ except for their last step, which we require to be an east
step ending at $(mn+1,n)$. The cardinality of $\Dmngen$ is well known
to be given by the Fuss-Catalan number $\Cmngen =
\frac{1}{mn+1}\binom{mn+n}{n}$ (see~\cite{duchon,vonFuss}).
We often identify a path in $\Dmngen$ with a certain string of
$n$ north steps and $mn+1$ east steps, which may be translated
to a starting point other than the origin in some situations.

\begin{prop}[See~\cite{catalan,larcombe,tedford}]\label{prop:decomp}
  For $n,m>0$, any $\pi\in\Dmngen$ has a unique decomposition of the
  form
  \begin{equation}\label{eq:decomp}
    \N\pi_1\pi_2 \cdots \pi_{m+1}
  \end{equation}
    where each $\pi_j \in \Dmn{m}{p_j}$ for some $p_j\geq 0$ and $p_1
    + \cdots + p_{m+1} = n-1$.  Furthermore, if $G(\pi) =
    (g_0,g_1,\ldots,g_{n-1})$, then for any $i$ with $1\leq i\leq p_1
    + \cdots + p_j$, $g_i \geq m-j+2$.
\end{prop}

\begin{proof}
  Let $\pi\in \Dmn{m}{n}$. We note that $\pi$ begins on the line
  $x=my$ and ends on the line $x=my+1$. Consider any point $(x_0,y_0)$
  on $\pi$ lying on the line $x=my-b$ for some $b\in
  \mathbb{Z}$. If the next step in $\pi$ is a north step, then the
  next point on $\pi$ lies on the line $x=my-b-m$; if an
  east step, then the next point on $\pi$ lies on the line
  $x=my-b+1$.

  Since $\pi$ stays weakly above $x=my$ except at the end, it must
  begin with a north step. So the first point on the path after the
  starting point of $(0,0)$ is $(0,1)$, which is on the line $x=my-m$. 
As $\pi$ ends at the point $(mn+1,n)$, which is on the line $x=my+1$, it must
  visit each of the lines $x=my-i$ at least once, for each $i$ in the range
 $m\geq i\geq -1$.  Furthermore, its initial visit to $x=my-j$
  must occur before any visit to $x=my-i$ whenever $m\geq j >
  i\geq -1$. Letting $\N$ denote a north step and $\E$ an east step, it
  follows that $\pi$ has a unique factorization of the form
  $\N\pi_1\pi_2 \cdots \pi_{m+1}$ where for each $j$ with $1\leq j\leq
  m+1$, $\N\pi_1\pi_2 \cdots \pi_j$ is the shortest initial segment
  of $\pi$ ending on the line $x=my-m+j$. By construction,
  each $\pi_j\in \Dmn{m}{p_j}$ for some $p_j\geq 0$, as otherwise there
  would be a shorter initial segment ending on the line
  $x=my-m+j$.  Since $\pi$ ends at $(mn+1,n)$, it
  follows that $p_1 + \cdots + p_{m+1} = n-1$.

  The claimed inequalities satisfied by the area vector follow from the
  fact that the north steps in $\N\pi_1\cdots\pi_j$ (except the first one)
  all start weakly left of $x=my-m+j-1$. So the area vector entries for
  this part of the path satisfy $g_i=mi+1-x_i\geq mi+1-(mi-m+j-1)=m-j+2$.
\end{proof}

\begin{corollary}
  \begin{equation*}
    \Cmngen = \sum_{\substack{i_1+\cdots+i_{m+1} = n-1\\i_j \geq 0\text{ for } 1\leq j\leq m+1}} \Cmn{m}{i_1}\Cmn{m}{i_2}\cdots \Cmn{m}{i_{m+1}}.
  \end{equation*}
\end{corollary}

Let $\Dmnkgen$ denote the set of $k$-skeletal paths for parameters $m$ and $n$ 
(again augmenting each path with one east step at the end), 
and let $\Cmnk{m}{n}{k} = |\Dmnk{m}{n}{k}|$. 

\begin{theorem}\label{thm:kskel-enum}
For all $m,n \geq 1$ and $0 \leq k\leq n-1$, $\Cmngen = \Cmnkgen$
 via an explicit bijection.
\end{theorem}
\begin{proof}
  We consider $m$ fixed and prove the theorem by induction on $n$. 
For each $k$, $m$, and $n$, we define a map $\phimnkgen: \Dmngen
  \rightarrow \Dmnkgen$ that we prove to be a bijection. To avoid
  clutter, we will frequently suppress the dependencies of $\phimnkgen$
  on $m$, $n$ and $k$ and simply write $\phibare$.

  Given a path $\pi \in \Dmngen$, 
  find the unique decomposition (as in Proposition~\ref{prop:decomp})
  \begin{equation}\label{eq:factor-pi}
    \pi = \N\pi_1 \pi_2 \cdots \pi_{m+1}
  \end{equation}
  such that each $\pi_i \in \Dmn{m}{p_i}$ for some $p_i\geq 0$. Choose
  the maximum $s\in\{1,2,\ldots,m+2\}$ such that
  $p_1+\cdots+p_{s-1}\leq k$ (such $s$ exists since the sum on the
  left side is zero for $s=1$). Write $\rev{\sigma}$ for the stepwise
  reversal of a path $\sigma$.  We define $\phibare$ by
  \begin{equation}
    \phibare(\pi) = \begin{cases}
      \pi, & n=k+1,\\
      \rev{\pi_{s+1} \cdots \pi_{m+1}} \N \pi_1 \cdots \pi_{s-1}\pi_s, & n > k+1 \geq p_s,\\
      \rev{\pi_{s+1} \cdots \pi_{m+1}} \N \pi_1 \cdots \pi_{s-1}\phimnk{m}{p_s}{k}(\pi_s), & n > k+1 < p_s.
      \end{cases}
  \end{equation}

Since $p_1+\cdots+p_{m+1}=n-1$ and $0\leq k\leq n-1$, we have $s=m+2$
if and only if $k+1=n$. Hence, in the definition of $\phibare$ above,
and in the discussion below, we have $s\leq m+1$ whenever $n > k+1$.

For any $n$, when $k=n-1$ it is routine to check that $m$-Dyck paths
coincide with $k$-skeletal paths, so the identity map gives the
required bijection.  This case covers the base case $n=1$ of the
induction. From now on, fix $n>1$ and assume $\phimnk{m}{p}{k}$ is a
bijection for all $p<n$.

In the rest of this proof, define the \emph{level} of a point $(x,y)\in\Z^2$
to be $\lvl(x,y)=my-x$, which is the signed horizontal distance from $(x,y)$
to the line $x=my$. We rephrase Conditions~(P0), (P1), and (P2) for
$k$-skeletal paths (see Section~\ref{subsec:path-def}) in terms of levels.
Condition (P0) is automatically satisfied since all lattice paths considered here
start at the origin. Since $c=1$ here, Conditions~(P1) and~(P2) may be 
rewritten as:
\begin{itemize}
\item[(P1$'$)] The last $k+1$ north steps of the path start at
 levels $\geq 0$.
\item[(P2$'$)] There do not exist $k+1$ consecutive rows such 
that the north steps in those rows all start at levels $>0$.
\end{itemize}

Suppose $\pi_i$ (one of the factors in~\eqref{eq:factor-pi})
appears as a subword of some path and starts in that path at level $\ell$.
Then all steps of $\pi_i$ start at levels $\geq\ell$,
and the final east step of $\pi_i$ ends at level $\ell-1$.
Next suppose $\rev{\pi_i}$ appears as a subword of some path
and starts in that path at level $\ell$. Then the first step of $\rev{\pi_i}$
goes east to level $\ell-1$, all subsequent steps of $\rev{\pi_i}$
start at levels $<\ell$, and $\rev{\pi_i}$ ends at level $\ell-1$.

Using the above observations, consider the levels reached by various
subpaths of $\phibare(\pi)$ when $n > k+1$. First consider the shared
prefix $\rev{\pi_{s+1}\cdots\pi_{m+1}}\N$. If $s=m+1$ then
$\rev{\pi_{s+1}\cdots\pi_{m+1}}$ is the empty word.  If $s<m+1$,
$\rev{\pi_{s+1}\cdots\pi_{m+1}}$ is the concatenation of $m+1-s$
factors $\rev{\pi_{m+1}},\ldots,\rev{\pi_{s+1}}$, it starts at level
$0$ with an east step, it has all subsequent steps starting at negative
levels, and it ends at level $s-(m+1) < 0$. In either case, the terminal
north step of $\rev{\pi_{s+1}\cdots\pi_{m+1}}\N$ starts at level
$s-(m+1)\leq 0$ and ends at level $s-1\geq 0$; this north step is the first 
step in $\phibare(\pi)$ ending at a nonnegative level.

We now consider the subpath $\pi_1\cdots\pi_{s-1}$. If $s-1=0$ then it
is empty. Otherwise, by the previous two paragraphs we know it has all
its steps starting at positive levels and its final step ending at
level $0$. Finally, we consider the subpath of $\phibare(\pi)$
following $\pi_1\cdots\pi_{s-1}$. In the case of $k+1 \geq p_s$, the
final subpath $\pi_s$ has all its steps starting at levels $\geq 0$,
and ends at level $-1$.  In the case of $k+1 < p_s$, the final subpath
$\phimnk{m}{p_s}{k}(\pi_s)$ of $\phibare(\pi)$ starts at level $0$,
ends at level $-1$, and satisfies Conditions~(P1$'$) and~(P2$'$) by
the induction hypothesis.

We now prove that $\phibare(\pi)$ satisfies Condition~(P1$'$) when $n
> k+1$.  We know that $p_1+\cdots+p_s\geq k+1$, by definition of $s$.
In the case of $k+1 \geq p_s$, the last $k+1$ north steps in
$\phibare(\pi)$ all appear in the suffix $\pi_1\pi_2\cdots\pi_s$. By
the preceding analysis, all these north steps start at levels $\geq
0$, as needed.  In the case of $k+1<p_s$, the last $k+1$ north
steps in $\phibare(\pi)$ all appear in the suffix
$\phimnk{m}{p_s}{k}(\pi_s)$.  By induction hypothesis, these north
steps all start at levels $\geq 0$.

We now prove that $\phibare(\pi)$ satisfies Condition~(P2$'$) when $n
> k+1$.  To get a contradiction, assume there are $k+1$ rows in
$\phibare(\pi)$ where the north steps in those rows all start at
levels $>0$.  The last factor of $\phibare(\pi)$, namely $\pi_s$ when
$k+1 \geq p_s$ or $\phimnk{m}{p_s}{k}(\pi_s)$ when $k+1 < p_s$, either
has no north steps at all or has first north step starting at level
$0$.  The north steps violating Condition~(P2$'$) must either all
occur after this north step or all occur before it.  We rule out the
first possibility as follows.  When $k+1\geq p_s$ there are not enough
north steps in $\pi_s$ (following the first north step) to cause a
violation.  When $k+1 < p_s$ we reach the same conclusion by invoking
the induction hypothesis to see that $\phimnk{m}{p_s}{k}(\pi_s)$ has
no violation.  Next we rule out the possibility of a violation earlier
in $\phibare(\pi)$.  By the level analysis above, the $k+1$ violating
north steps must all occur in the subword $\pi_1\pi_2\cdots \pi_{s-1}$
of $\phibare(\pi)$.  But, regardless of how $k+1$ compares to $p_s$,
$p_1+\cdots+p_{s-1}<k+1$ by definition of $s$, so there are not enough
available north steps in this region to cause a violation.

So far, we have proved that $\phimnkgen$ maps $\Dmngen$ into $\Dmnkgen$. 
It remains to show that $\phimnkgen$ is a bijection. 
We define an inverse $\psibare=\psimnkgen: \Dmnkgen
  \rightarrow \Dmngen$ recursively as follows.
Decompose $\omega\in\Dmnkgen$ as $\omega = \omega_1 \N \omega_2
  \omega_3$ where:
  \begin{itemize}
    \item $\omega_1 \N$ is the shortest initial segment of $\omega$
      ending at a nonnegative level; 
    \item $\omega_1 \N \omega_2$ is the shortest initial segment 
      ending at level zero. 
  \end{itemize}
  Note that $\omega_1$ is either empty or begins with an east step and
  that $\omega_2$ may be empty. Let $p$ be the number of north steps 
 in $\omega_3$. We define $\psibare$ using three cases.
 \begin{equation*}
   \psibare(\omega) = \begin{cases}
     \omega, & n=k+1,\\
     \N\omega_2\omega_3\rev{\omega_1}, & n>k+1\geq p,\\
     \N\omega_2 \psimnk{m}{p}{k}(\omega_3)\rev{\omega_1}, & n > k+1 < p.     
   \end{cases}
\end{equation*}
By a level analysis similar to what appears above, one may show that:
$\omega_3$ belongs to $\Dmnk{m}{p}{k}$ when $n > k+1 < p$;
$\psibare$ does map $\Dmnkgen$ into $\Dmngen$;
and $\psibare$ is the two-sided inverse of $\phibare$.
We omit these details. 
\end{proof}

\begin{example}\label{ex:phi}
 Consider $k=1$ and
 \begin{equation*}
  \pi =\mathrm{N}\overbracket[0.5pt]{\mathrm{NEEE}}^{\pi_1}\overbracket[0.5pt]{\mathrm
{E}}^{\pi_2}
 \overbracket[0.5pt]{\mathrm{NNENEEENNEEENEEEEEE}}^{\pi_3}\in\Dmn{2}{8},
 \end{equation*}
 as illustrated in Figure~\ref{fig:phi}(a). Since $s=3$, we find that
 $\phibare(\pi) =
 \N\pi_1\pi_2\phibare(\pi_3)=\mathrm{NNEEEE}\phibare(\pi_3)$. We now
 decompose $\pi_3$ as
 \begin{equation*}
  \pi_3 = \mathrm{N} \overbracket[0.5pt]{\mathrm{NENEEENNEEENEEEE}}^{\pi_1'}
 \overbracket[0.5pt]{\E}^{\pi_2'} \overbracket[0.5pt]{\E}^{\pi_3'},
 \end{equation*}
 from which it follows that $s'=1$ and that
 \begin{equation*}
 \phibare(\pi_3) = \rev{\pi_2'\pi_3'}\N\phibare(\pi_1') =
\mathrm{EEN}\phibare(\pi_1').
 \end{equation*}
 Then we decompose $\pi_1'$ as
 \begin{equation*}
  \pi_1' = \N \overbracket[0.5pt]{\E}^{\pi_1''}
\overbracket[0.5pt]{\mathrm{NEEE}}^{\pi_2''}
\overbracket[0.5pt]{\mathrm{NNEEENEEEE}}^{\pi_3''},
 \end{equation*}
 so $s''=3$ and $\phibare(\pi_1') = \N \pi_1'' \pi_2'' \phibare(\pi_3'') = \N
\mathrm{E} \mathrm{NEEE} \phibare(\pi_3'')$.
 We now decompose $\pi_3''$ as
 \begin{equation*}
  \pi_3'' =\mathrm{N}\overbracket[0.5pt]{\mathrm{NEEE}}^{\pi_1'''}\overbracket[0.5pt]{\mathrm{NEEE}}^{\pi_2'''}
 \overbracket[0.5pt]{\mathrm{E}}^{\pi_3'''},
 \end{equation*}
 from which it follows that $s'''=2$ and $\phibare(\pi_3'') =
\rev{\pi_3'''}\N\pi_1'''\phibare(\pi_2''') =
\mathrm{ENNEEE}\phibare(\pi_2''')$. Finally, when we decompose $\pi_2'''$
we find that $s'''' = 4$, so $\phibare(\pi_2''')=\pi_2'''=\mathrm{NEEE}$.
 Combining our results, we conclude that
 \begin{align*}
 \phibare(\pi) &= \mathrm{NNEEEE}\phibare(\pi_3)\\
 &= \mathrm{NNEEEE}\,\mathrm{EEN}\phibare(\pi_1')\\
 &= \mathrm{NNEEEE}\,\mathrm{EEN}\,\mathrm{NENEEE}\,\phibare(\pi_3'')\\
 &=\mathrm{NNEEEE}\,\mathrm{EEN}\,\mathrm{NENEEE}\,\mathrm{ENNEEE}\phibare(\pi_2''')\\
 &=\mathrm{NNEEEE}\,\mathrm{EEN}\,\mathrm{NENEEE}\,\mathrm{ENNEEE}\,\mathrm{NEEE},
 \end{align*}
 as illustrated in Figure~\ref{fig:phi}(b).
\end{example}
\begin{figure}[h]
\begin{center}
  \includegraphics[width=0.9\linewidth]{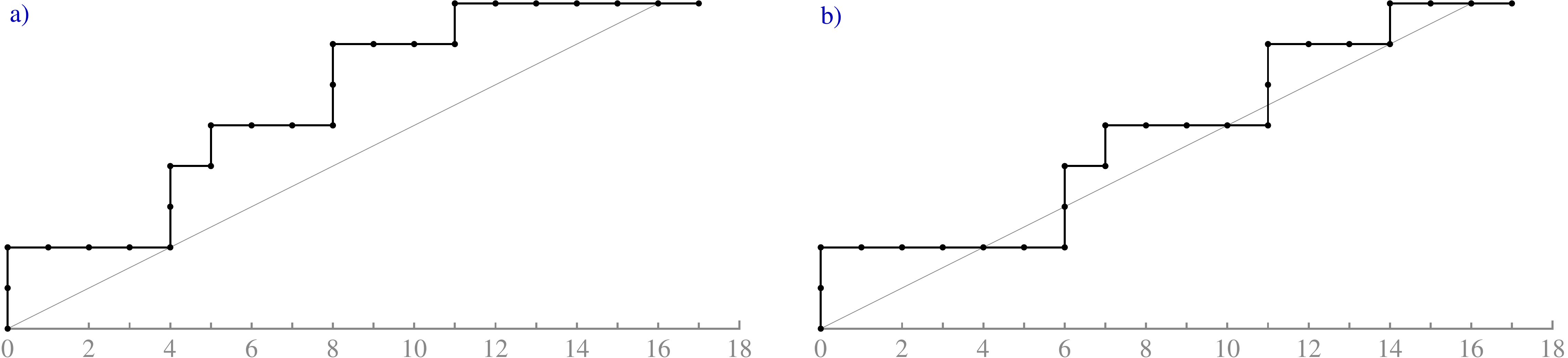}
  \caption{(a) The path $\pi$ from Example~\ref{ex:phi} and (b) its
    image under $\phibare$. Note that we are not showing the reference
    line $x=my+c$, but rather the standard diagonal $x=my$.}
\label{fig:phi}
\end{center}
\end{figure}

\section{Acknowledgments}
 S.B. thanks for Anton Dochtermann, Sam Hopkins, Lionel Levine, Alex
 Postnikov, and Christian Stump for discussion about related
 constructions. We thank Richard Stanley and Catherine Yan for their
 encouragement with this project.


\bibliography{sk080924}{}

\begin{thebibliography}{10}

\bibitem{amini2010riemann}
Omid Amini and Madhusudan Manjunath.
\newblock Riemann-{R}och for sub-lattices of the root lattice ${A}_n$.
\newblock {\em The Electronic Journal of Combinatorics}, pages R124--R124,
  2010.

\bibitem{an2014canonical}
Yang An, Matthew Baker, Greg Kuperberg, and Farbod Shokrieh.
\newblock Canonical representatives for divisor classes on tropical curves and
  the {M}atrix--{T}ree {T}heorem.
\newblock In {\em Forum of Mathematics, Sigma}, volume~2, page e24. Cambridge
  University Press, 2014.

\bibitem{armstrong-ratcat}
Drew Armstrong, Nicholas~A. Loehr, and Gregory~S. Warrington.
\newblock Rational parking functions and {C}atalan numbers.
\newblock {\em Annals of Combinatorics}, 20:21--58, 2016.

\bibitem{armstrong2013rational}
Drew Armstrong, Brendon Rhoades, and Nathan Williams.
\newblock Rational {C}atalan combinatorics: {T}he associahedron.
\newblock {\em Discrete Mathematics \& Theoretical Computer Science}, DMTCS
  Proceedings vol. AS, 25th International Conference on Formal Power Series and
  Algebraic Combinatorics (FPSAC 2013):933--944, 2013.

\bibitem{bacher1997lattice}
Roland Bacher, Pierre~de La~Harpe, and Tatiana Nagnibeda.
\newblock The lattice of integral flows and the lattice of integral cuts on a
  finite graph.
\newblock {\em Bulletin de la soci{\'e}t{\'e} math{\'e}matique de France},
  125(2):167--198, 1997.

\bibitem{Backman-bij}
Spencer Backman.
\newblock A bijection between the recurrent configurations of a hereditary
  chip-firing model and spanning trees.
\newblock {\em arXiv preprint arXiv:1207.6175}, 2012.

\bibitem{bwl-forthcoming}
Spencer Backman, Nicholas~A. Loehr, and Gregory~S. Warrington.
\newblock Quantized chip firing.
\newblock In preparation.

\bibitem{bak}
Per {Bak}, Chao {Tang}, and Kurt {Wiesenfeld}.
\newblock {Self-organized criticality: An explanation of the $1/f$ noise}.
\newblock {\em Physical Review Letters}, 59(4):381--384, Jul 1987.

\bibitem{baker-norine}
Matthew Baker and Serguei Norine.
\newblock Riemann--{R}och and {A}bel--{J}acobi theory on a finite graph.
\newblock {\em Adv. Math.}, 215(2):766--788, 2007.

\bibitem{baker2013chip}
Matthew Baker and Farbod Shokrieh.
\newblock Chip-firing games, potential theory on graphs, and spanning trees.
\newblock {\em Journal of Combinatorial Theory, Series A}, 120(1):164--182,
  2013.

\bibitem{benson}
Brian Benson, Deeparnab Chakrabarty, and Prasad Tetali.
\newblock {$G$}-parking functions, acyclic orientations and spanning trees.
\newblock {\em Discrete Mathematics}, 310(8):1340--1353, 2010.

\bibitem{biggs}
N.~Biggs.
\newblock Chip-firing and the critical group of a graph.
\newblock {\em J. of Alg. Comb.}, 9:25--45, 1999.

\bibitem{bjorner-lovasz-shor}
Anders Bj\"orner, L\'aszl\'o Lov\'asz, and Peter~W. Shor.
\newblock Chip-firing games on graphs.
\newblock {\em European Journal of Combinatorics}, 12(4):283--291, 1991.

\bibitem{caracciolo2012multiple}
Sergio Caracciolo, Guglielmo Paoletti, and Andrea Sportiello.
\newblock Multiple and inverse topplings in the abelian sandpile model.
\newblock {\em The European Physical Journal Special Topics}, 212(1):23--44,
  2012.

\bibitem{catalan}
M.~E. Catalan.
\newblock Sur les nombres de {S}egner.
\newblock {\em Rend. Circ. Mat. Palermo}, 1:190--201, 1887.

\bibitem{cayley}
Arthur {C}ayley.
\newblock A theorem on trees.
\newblock {\em Quarterly Journal of Mathematics}, 23:376--378, 1889.

\bibitem{cori2003sand}
Robert Cori and Yvan Le~Borgne.
\newblock The sand-pile model and {T}utte polynomials.
\newblock {\em Advances in Applied Mathematics}, 30(1-2):44--52, 2003.

\bibitem{cori2016hall}
Robert Cori, Pasquale Petrullo, and Domenico Senato.
\newblock Hall--{L}ittlewood symmetric functions via the chip-firing game.
\newblock {\em European Journal of Combinatorics}, 58:225--237, 2016.

\bibitem{rossin}
Robert Cori, Dominique Rossin, and Bruno Salvy.
\newblock Polynomial ideals for sandpiles and their {G}röbner bases.
\newblock {\em Theoretical Computer Science}, 276(1):1--15, 2002.

\bibitem{corry2018divisors}
Scott Corry and David Perkinson.
\newblock {\em Divisors and sandpiles}, volume 114.
\newblock American Mathematical Soc., 2018.

\bibitem{Sandpile2024}
Michele D'Adderio, Mark Dukes, Alessandro Iraci, Alexander Lazar, Yvan~Le
  Borgne, and Anna~Vanden Wyngaerd.
\newblock Shuffle theorems and sandpiles.
\newblock {\em arXiv preprint arXiv:2401.06488}, 2024.

\bibitem{Dhar}
Deepak Dhar.
\newblock Self-organized critical state of sandpile automaton models.
\newblock {\em Phys. Rev. Lett.}, 64:1613--1616, 1990.

\bibitem{Dochtermann}
Anton Dochtermann.
\newblock One-skeleta of {$G$}-parking function ideals: resolutions and
  standard monomials, 2017.

\bibitem{dochtermann2021trees}
Anton Dochtermann and Westin King.
\newblock Trees, parking functions, and standard monomials of skeleton ideals.
\newblock {\em Australasian Journal of Combinatorics}, 81(1):126--151, 2021.

\bibitem{duchon}
Philippe Duchon.
\newblock On the enumeration and generation of generalized {D}yck words.
\newblock {\em Discrete Mathematics}, 225(1):121--135, 2000.
\newblock FPSAC'98.

\bibitem{carlsson}
{E}rik {C}arlsson and {A}nton {M}ellit.
\newblock A proof of the {S}huffle {C}onjecture.
\newblock {\em Journal of the American Mathematical Society}, 31(3):pp.
  661--697, 2018.

\bibitem{fur-hof}
J.~{F}\"urlinger and J.~{H}ofbauer.
\newblock $q$-{C}atalan numbers.
\newblock {\em Journal of Combinatorial Theory, Series A}, 40(2):248--264,
  1985.

\bibitem{gaydarov2016parking}
Petar Gaydarov and Sam Hopkins.
\newblock Parking functions and tree inversions revisited.
\newblock {\em Advances in Applied Mathematics}, 80:151--179, 2016.

\bibitem{gorsky-mazin-ii}
E.~Gorsky and M.~Mazin.
\newblock Compactified {J}acobians and $q,t$-{C}atalan numbers, {II}.
\newblock {\em J. Alg. Comb.}, 39:153--186, 2014.

\bibitem{hag-mac-conj}
J.~Haglund.
\newblock A combinatorial model for the {M}acdonald polynomials.
\newblock {\em Proceedings of the National Academy of Sciences},
  101(46):16127--16131, 2004.

\bibitem{hag-book}
J.~Haglund.
\newblock {\em The $q,t$-{C}atalan Numbers and the Space of Diagonal Harmonics:
  With an Appendix on the Combinatorics of {M}acdonald Polynomials}.
\newblock University lecture series. American Mathematical Soc., 2008.

\bibitem{hhlru}
James Haglund, Mark~D. Haiman, Nicholas~A. Loehr, Jeffrey~B. Remmel, and
  A.~Ulyanov.
\newblock A combinatorial formula for the character of the diagonal
  coinvariants.
\newblock {\em Duke Mathematical Journal}, 126:195--232, 2003.

\bibitem{klivans}
Caroline Klivans.
\newblock {\em The Mathematics of Chip-firing}.
\newblock Chapman \& Hall, 11 2018.

\bibitem{konheim}
Alan~G. Konheim and Benjamin Weiss.
\newblock An occupancy discipline and applications.
\newblock {\em Siam Journal on Applied Mathematics}, 14:1266--1274, 1966.

\bibitem{kostic2008multiparking}
Dimitrije Kosti{\'c} and Catherine~H Yan.
\newblock Multiparking functions, graph searching, and the {T}utte polynomial.
\newblock {\em Advances in Applied Mathematics}, 40(1):73--97, 2008.

\bibitem{kreweras}
G.~Kreweras.
\newblock Une famille de polyn\^omes ayant plusieurs propri\'et\'es
  \'enumeratives.
\newblock {\em Periodica Mathematica Hungarica}, 11(4):309--320, 1980.

\bibitem{kumar2021skeleton}
Chanchal Kumar and Gargi Lather.
\newblock Skeleton ideals of certain graphs, standard monomials and spherical
  parking functions.
\newblock {\em The Electronic Journal of Combinatorics}, pages P1--53, 2021.

\bibitem{kumar2022standard}
Chanchal Kumar, Gargi Lather, and Amit Roy.
\newblock Standard monomials of 1-skeleton ideals of graphs and generalized
  signless {L}aplacians.
\newblock {\em Linear Algebra and its Applications}, 637:24--48, 2022.

\bibitem{larcombe}
Peter Larcombe and David French.
\newblock The {C}atalan number $k$-fold self-convolution identity: the original
  formulation.
\newblock {\em JCMCC. The Journal of Combinatorial Mathematics and
  Combinatorial Computing}, 46, 01 2003.

\bibitem{levine-peres}
L.~Levine and Y.~Peres.
\newblock Strong spherical asymptotics for rotor-router aggregation and the
  divisible sandpile.
\newblock {\em Potential Anal.}, 30:1--27, 2009.

\bibitem{loehr-comb}
N.~Loehr.
\newblock {\em Combinatorics}.
\newblock CRC Press, second edition, 2017.

\bibitem{loehr-trap}
Nicholas~A. Loehr.
\newblock Trapezoidal lattice paths and multivariate analogues.
\newblock {\em Advances in Applied Mathematics}, 31(4):597--629, 2003.

\bibitem{loehr-mcat}
Nicholas~A. Loehr.
\newblock Conjectured statistics for the higher $q,t$-{C}atalan sequences.
\newblock {\em Electron. J. Comb.}, 12, 2005.

\bibitem{Lorenzini89}
Dino~J. Lorenzini.
\newblock Arithmetical graphs.
\newblock {\em Math. Ann.}, 285(3):481--501, 1989.

\bibitem{Lorenzini91}
Dino~J. Lorenzini.
\newblock A finite group attached to the {L}aplacian of a graph.
\newblock {\em Discrete Math.}, 91(3):277--282, 1991.

\bibitem{manjunath2013monomials}
Madhusudan Manjunath and Bernd Sturmfels.
\newblock Monomials, binomials and {R}iemann--{R}och.
\newblock {\em Journal of Algebraic Combinatorics}, 37:737--756, 2013.

\bibitem{mikhalkin-zharkov}
G.~Mikhalkin and I.~Zharkov.
\newblock {\em Tropical curves, their {J}acobians and theta functions.}, pages
  203--230.
\newblock Amer. Math. Soc., Providence, RI, 2008.

\bibitem{mosesian1972strongly}
KM~Mosesian.
\newblock Strongly basable graphs.
\newblock In {\em Akad. Nauk. Armian. SSR Dokl}, volume~54, pages 134--138,
  1972.

\bibitem{perkinson}
David Perkinson, Jacob Perlman, and John Wilmes.
\newblock Primer for the algebraic geometry of sandpiles.
\newblock In {\em Tropical and non-{A}rchimedean geometry}, volume 605 of {\em
  Contemp. Math.}, pages 211--256. Amer. Math. Soc., Providence, RI, 2013.

\bibitem{perkinson2017g}
David Perkinson, Qiaoyu Yang, and Kuai Yu.
\newblock {G}-parking functions and tree inversions.
\newblock {\em Combinatorica}, 37:269--282, 2017.

\bibitem{postnikov-shapiro}
Alexander Postnikov and Boris Shapiro.
\newblock Trees, parking functions, syzygies, and deformations of monomial
  ideals.
\newblock {\em Trans. Amer. Math. Soc.}, 356(8):3109--3142, 2004.

\bibitem{pretzel1986reorienting}
Oliver Pretzel.
\newblock On reorienting graphs by pushing down maximal vertices.
\newblock {\em Order}, 3:135--153, 1986.

\bibitem{propp2002lattice}
James Propp.
\newblock Lattice structure for orientations of graphs.
\newblock {\em arXiv preprint math/0209005}, 2002.

\bibitem{raynaud1970specialisation}
Michel Raynaud.
\newblock Sp{\'e}cialisation du foncteur de {P}icard.
\newblock {\em Publications Math{\'e}matiques de l'IH{\'E}S}, 38:27--76, 1970.

\bibitem{riordan}
John Riordan.
\newblock Ballots and trees.
\newblock {\em Journal of Combinatorial Theory}, 6(4):408--411, 1969.

\bibitem{roy2020standard}
Amit Roy.
\newblock Standard monomials of $1$-skeleton ideals of multigraphs.
\newblock {\em arXiv preprint arXiv:2010.14474}, 2020.

\bibitem{addendum}
Richard~P. Stanley.
\newblock Catalan {A}ddendum.
\newblock \url{https://math.mit.edu/~rstan/ec/catadd.pdf}. Version of May 25,
  2013; accessed Aug 09, 2024.

\bibitem{ECII}
Richard~P. Stanley.
\newblock {\em Enumerative Combinatorics}.
\newblock Cambridge Studies in Advanced Mathematics. Cambridge University
  Press, 1999.

\bibitem{rstan-park}
Richard~P. Stanley and Yinghui Wang.
\newblock Some aspects of $(r,k)$-parking functions.
\newblock {\em Journal of Combinatorial Theory, Series A}, 159:54--78, 2018.

\bibitem{tedford}
Steven~J. Tedford.
\newblock Combinatorial interpretations of convolutions of the {C}atalan
  numbers.
\newblock {\em Integers}, 11(1):35--45, 2011.

\bibitem{vonFuss}
N.~von Fuss.
\newblock Solutio quaestionis quot modis polygonum $n$ laterum in polygona $m$
  laterum per diagonales resolvi queat.
\newblock {\em Nova acta Academiae scientiarum imperialis petropolitanae},
  IX:243--251, 1795.

\bibitem{yuen}
Chi~Ho Yuen.
\newblock Geometric bijections between spanning trees and break divisors.
\newblock {\em J. Combin. Theory Ser. A}, 152:159--189, 2017.

\end{thebibliography}
\bibliographystyle{plain}


\end{document}